\documentclass[preprint,3p]{elsarticle}

\usepackage{graphicx}
\usepackage{amsmath,amssymb}
\usepackage{color,amscd,amsthm}
\usepackage{tikz}
\usepackage{float}
\usetikzlibrary{arrows, intersections}
\pgfdeclarelayer{bg}
\pgfdeclarelayer{fg}
\pgfdeclarelayer{crossing over}
\pgfsetlayers{main,bg,crossing over,fg}
\usepackage{hyperref}

\definecolor{brown}{rgb}{0.43, 0.21, 0.1}
\definecolor{green}{cmyk}{0.64,0,0.95,0.40}

\def\ml{l\kern-0.035cm\char39\kern-0.03cm}

\numberwithin{equation}{section}

\newtheorem{theorem}{Theorem}[section]
\newtheorem*{theorem*}{Theorem}
\newtheorem{lema}[theorem]{Lemma}

\newtheorem{corol}[theorem]{Corollary}
\newtheorem{prop}[theorem]{Proposition}

\newtheorem{quest}{Question}
\newtheorem{clma}{Claim}

\newtheorem*{clm*}{Claim}
\newtheorem{dfn}{Definition}
\newtheorem{problem}{Problem}

\newcommand{\oo}{\text{\rm 1-1}}
\newcommand{\KB}{\text{\rm KB}}

\newcommand{\CH}{\text{\rm\bf CH}}

\newcommand{\rmL}{\mathrm{L}}

\newcommand{\rmP}{\mathrm{P}}

\newcommand{\leqsq}{\leq_\square}
\newcommand{\leqoo}{\leq_\oo}
\newcommand{\leqkb}{\leq_\KB}
\newcommand{\leqk}{\leq_K}
\newcommand{\leqcountk}{\leq_K^\omega}
\newcommand{\leqrestk}{\leq_K^+}

\newcommand{\baire}{{}^{\omega}}
\newcommand{\cp}[1]{\mathrm{C}_p(#1)}
\newcommand{\Fsigma}{\mathrm{F}_\sigma}

\newcommand{\Liminf}{\underline{\mathrm{Lim}}}

\newcommand{\partlim}{\mathtt{partlim}}

\newcommand{\fup}{{\rm\texttt{FUP}}}

\newcommand{\fip}{{\rm\texttt{FIP}}}
\newcommand{\tall}{{\text{tall}}}

\newcommand{\zero}{{\mathbf 0}}
\newcommand{\bi}{{\mathbf i}}

\newcommand{\R}{{\mathbb R}}
\newcommand{\Q}{{\mathbb Q}}
\newcommand{\fin}{\mathtt{Fin}}
\newcommand{\infin}{[\omega]^\omega}
\newcommand{\covh}[1]{\mathtt{cov}^*(#1)}

\newcommand{\seqn}[2]{\langle #1 :\ #2\in \omega \rangle}
\newcommand{\gseqn}[2]{{\langle #1:\ #2\rangle}}
\newcommand{\set}[2]{\{#1:\ #2\}}

\newcommand{\I}{{\mathcal I}}
\newcommand{\J}{{\mathcal J}}
\newcommand{\K}{{\mathcal K}}
\newcommand{\calS}{{\mathcal S}}
\newcommand{\U}{{\mathcal U}}
\newcommand{\V}{{\mathcal V}}
\newcommand{\W}{{\mathcal W}}
\newcommand{\A}{{\mathcal A}}

\newcommand{\F}{{\mathcal F}}
\newcommand{\G}{{\mathcal G}}
\newcommand{\B}{{\mathcal B}}

\newcommand{\PP}{{\mathcal P}}

\newcommand{\Ran}{\mathtt{Ran}}
\newcommand{\ED}{{\mathcal ED}}
\newcommand{\EDfin}{{\mathcal ED}_\text{\rm fin}}
\newcommand{\conv}{\mathtt{conv}}
\newcommand{\nwd}{\mathtt{nwd}}
\newcommand{\finfin}{\fin\times\fin}
\newcommand{\Z}{{\mathcal Z}}
\newcommand{\summable}{\I_{1/n}}

\newcommand{\bb}{\mathfrak{b}}
\newcommand{\dd}{\mathfrak{d}}
\newcommand{\pp}{\mathfrak{p}}
\newcommand{\cc}{\mathfrak{c}}

\newcommand{\ppk}[1]{\pp_\mathrm{K}(#1)}
\newcommand{\ppkb}[1]{\pp_\mathrm{KB}(#1)}

\newcommand{\ppsq}[1]{\pp_\square(#1)}

\newcommand{\cov}{\mathtt{cov}}
\newcommand{\covm}{\cov(\mathcal{M})}

\newcommand{\non}[1]{\mathtt{non}(#1)}

\newcommand{\iGamma}[1]{#1\text{-}\Gamma}

\newcommand{\pytkeev}{\texttt{Pytkeev}}

\newcommand{\sone}[2]{{\rm S}_1(#1,#2)}

\newcommand{\sonemm}{\sone{\Omega}{\Omega}}

\newcommand{\subsel}[2]{{#1\brack#2}}
\newcommand{\subselR}[2]{\left[#1,#2\right]}
\newcommand{\subselgg}[2]{\subsel{#1\text{-}\Gamma}{#2\text{-}\Gamma}}
\newcommand{\subselggR}[2]{\subselR{#1\text{-}\Gamma}{#2\text{-}\Gamma}}

\newcommand{\subselggsqR}[2]{\subselR{#1\text{-}\Gamma}{#2\text{-}\Gamma}_\square}

\newcommand{\subselggf}[1]{\subsel{#1\text{-}\Gamma}{\Gamma}}
\newcommand{\subselggfR}[1]{\subselR{#1\text{-}\Gamma}{\Gamma}}

\newcommand{\subselmg}[1]{\subsel{\Omega}{#1\text{-}\Gamma}}
\newcommand{\subselmgR}[1]{\subselR{\Omega}{#1\text{-}\Gamma}}

\newcommand{\subselmgsqR}[1]{\subselR{\Omega}{#1\text{-}\Gamma}_\square}
\newcommand{\subselmgk}[1]{\subsel{\Omega}{#1\text{-}\Gamma}_K}
\newcommand{\subselmgkR}[1]{\subselR{\Omega}{#1\text{-}\Gamma}_K}
\newcommand{\subselmgkb}[1]{\subsel{\Omega}{#1\text{-}\Gamma}_\KB}
\newcommand{\subselmgkbR}[1]{\subselR{\Omega}{#1\text{-}\Gamma}_\KB}
\newcommand{\subselmgoo}[1]{\subsel{\Omega}{#1\text{-}\Gamma}_\oo}
\newcommand{\subselmgooR}[1]{\subselR{\Omega}{#1\text{-}\Gamma}_\oo}

\newcommand{\zsubselmg}[1]{\subsel{\Omega_\zero}{#1\text{-}\Gamma_\zero}}
\newcommand{\zsubselmgR}[1]{\subselR{\Omega_\zero}{#1\text{-}\Gamma_\zero}}

\newcommand{\zsubselmgsqR}[1]{\subselR{\Omega_\zero}{#1\text{-}\Gamma_\zero}_\square}

\newcommand{\fu}{\rm\texttt{FU}}

\input cyracc.def

\journal{arXiv.org}

\begin{document}

\begin{frontmatter}
\title{Ideal approach to convergence in functional spaces} 
\author{Serhii Bardyla\fnref{fn1}}
\address{Universit\"at Wien, Institut f\"ur Mathematik, Kurt G\"odel Research Center, Kolingasse 14-16, 1090 Vienna, Austria.}
\ead{sbardyla@yahoo.com}
\author{Jaroslav \v Supina\fnref{fn2}}
 \address{Institute of Mathematics,  P.J. \v{S}af\'arik University in Ko\v sice,  Jesenn\'a 5, 040 01 Ko\v{s}ice, Slovakia}
\ead{jaroslav.supina@upjs.sk}
\author{Lyubomyr Zdomskyy\fnref{fn3}}
\address{Universit\"at Wien, Institut f\"ur Mathematik, Kurt G\"odel Research Center, Kolingasse 14-16, 1090 Vienna, Austria.}
\ead{lzdomsky@gmail.com}
\fntext[fn1]{Supported by the Austrian Science Fund FWF (Grant M 2967).}
\fntext[fn2]{The~author would like to thank the~Austrian Agency for International Cooperation in Education and Research (OeAD-GmbH) for the~scholarship ICM-2020-00442 in the~frame of~Aktion Österreich-Slowakei, AÖSK-Stipendien für Postdoktoranden. This work was supported by the Slovak Research and Development Agency under the Contracts no. APVV-16-0337, APVV-20-0045.}
\fntext[fn3]{Supported by the~Austrian Science Fund FWF (Grants I 2374, I 3709, and I 5930).}
\begin{abstract}

We solve the~last standing open problem from the~seminal paper by~J.~Gerlits and Zs.~Nagy~\cite{GerNag}, which was later reposed by A.~Miller, T.~Orenshtein, and B.~Tsaban. Namely, we show that under~$\pp=\cc$ there is a~$\delta$-set that is not a~$\gamma$-set. 
 Thus we constructed a subset $A$ of reals such that the space $\cp{A}$ of all real-valued continuous functions on $A$ is not Fr\' echet--Urysohn, but possesses the Pytkeev property.
Moreover, under~$\CH$ we construct a~$\pi$-set that is not a~$\delta$-set solving a~problem by M.~Sakai. In fact, we construct various examples of $\delta$-sets that are not $\gamma$-sets, satisfying finer properties parametrized by ideals on natural numbers. Finally, we distinguish ideal variants of the~Fr\' echet--Urysohn property for many different Borel ideals in the~realm of functional spaces. 
\end{abstract}
\end{frontmatter}

\section{Introduction}\label{S-intro}

This paper is devoted to the study of the \emph{Fr\'echet--Urysohn} property 
of  space of the form $\cp{X}$, the set of all 
continuous functions $f:X\to\mathbb R$ with the topology inherited from the Tychonoff power $\mathbb R^X$, where $X$ is a Tychonoff space. 
The Fr\'echet--Urysohn property is one of the two most popular approaches
to characterize spaces which can be described completely by the family of their convergent sequences, the other property being what is now called \emph{sequentiality}, 
see, e.g., \cite{Dud64, Fra65} and references therein. As it turned out \cite{Pyt82_n}, these two properties coincide for spaces of the form $\cp{X}$.

We study the Fr\'echet--Urysohn property of $\mathrm{C}_p$-spaces by 
investigating its relations with formally weaker local properties.
The first one has been introduced by Pytkeev in \cite{Pytk84}
and since at least \cite{MaTi} is named after him. The~second one 
is called \emph{partlim} in~\cite{OrTs13}, which is intuitively close to the~sequentiality, but expressed in terms of partial functions. 
 
 Our main conceptual approach is to use the 
 duality given by the correspondence 
between the local properties  of $\mathrm{C}_p$-spaces 
and global (covering) properties of the ``base'' space $X$, see, e.g.,
\cite{Ark_book} for many classical examples of this duality. 
For instance, $\cp{X}$ is Fr\'echet--Urysohn (resp. Pytkeev, satisfies ``partlim'') if and only if $X$ is a $\gamma$-space (resp. $\pi$-space, $\delta$-space), see \cite{GerNag} (resp. \cite{Sak03,SimTsa}, \cite{OrTs13}).
This allows us to distinguish the above-mentioned properties of $\mathrm{C}_p$-spaces
by proving that the characterizing covering properties of their base spaces
are consistently different, and for this we  shall use the technical tools from the well-developed by now theory of
\emph{selection principles in topology} taking their origin in \cite{COC2,Comb1}, see also
  \cite{buk-str,Bu19,Os18} for  more recent sources on the~topic.  It was our goal to get such counterexamples of as 
``simple'' form as possible, which in this context means for  
$X$ to be a zero-dimensional metrizable separable space, i.e., a topological copy of a subspace of the Cantor space $2^\omega$. For subspaces of $2^\omega$ it is common to refer to
$\xi$-spaces as to $\xi$-sets, where $\xi\in\{\gamma, \delta,\pi\}$.

Next, we give a short road map of the paper. 
In Section~\ref{S-nongamma_pytkeev} we show that the equality $\mathfrak p=\mathfrak c$ implies the existence of a $\pi$-set that is not a $\gamma$-set.
In Section~\ref{S-nongamma_delta} we show that the same set-theoretic assumption yields a $\delta$-set
that is not a $\gamma$-set. 
Let us mention that every
$\gamma$-space is a $\delta$-space simply by the corresponding definitions, and each $\delta$-space is a $\pi$-space, see \cite{SakQuad}. 
Thus the main result of Section~\ref{S-nongamma_pytkeev} follows from that of Section~\ref{S-nongamma_delta}.
However, since our argument in Section~\ref{S-nongamma_pytkeev} is direct in the sense 
that it does not use the fact that $\delta$-sets are $\pi$-sets, we believe it might be still of some interest.
The proofs in these initial two sections are rather streamlined, and it was our intention to begin with relatively easy negative solutions of corresponding   problems, since we suspect that some readers will  primarily 
be interested  in exactly these counterexamples. Thus, the readers who would like just to get a taste of the subject without going through technically involved arguments could simply  read only the first three sections.

A large part of the rest of the paper is motivated by a question of M.~Sakai \cite{SakQuad} asking whether 
$\pi$-sets are $\delta$-sets. Unlike it was with questions treated in Sections \ref{S-nongamma_pytkeev} and \ref{S-nongamma_delta},
our solution of this one required the introduction of ideals on $\omega$ as parameters into the study of weaker versions of $\gamma$-sets in the style of~\cite{Su20}, which makes the paper starting  from  Section~\ref{S-katetov}  considerably more technical. 
In particular,  in Sections~\ref{S-katetov}  and \ref{S-delta_hier} we essentially develop the machinery needed to construct under CH a $\pi$-set that is not a $\delta$-set, see Theorem~\ref{pi-nondelta} in Section~\ref{S-pytkeev_nondelta}.
It is worth noting here that the counterexamples mentioned above could not be obtained in ZFC because 
in the Laver model all $\pi$-sets (and hence also all $\delta$- and $\gamma$-sets) are countable \cite{SimTsa}.

These parametrized properties allow us to explore in more detail the~gap between $\delta$-sets and $\gamma$-sets. In fact, using countable powers of the ideal of finite subsets of $\omega$, we define $\omega_1$ many  topological properties between being $\gamma$- and $\delta$-sets, see Diagram~\ref{diagram:1}, which are all different under CH according to
Corollary~\ref{finalpha-nonfinbeta}.

The examples of spaces satisfying parametrized properties corresponding to ideals on $\omega$  are  concentrated on  $[\omega]^{<\omega}$ in a strong sense, and the ideals we consider are meager (in most of the cases even Borel), hence their function spaces  
satisfy analogous parametrized versions of the Fr\'echet--Urysohn property considered in \cite{BorFar,Su20}, see 
Section~\ref{S-frecheturysohn}. In a sense, this closes the circle: 
The investigation of the weakenings of the  Fr\'echet--Urysohn property of $\mathrm{C}_p$-spaces led us to the study of ideal versions of  $\gamma$-spaces, whose $\mathrm{C}_p$-spaces turned out to satisfy corresponding  ideal versions of the Fr\'echet--Urysohn property, thus allowing us to distinguish also between many of the latter ones, see Corollary~\ref{9_6_n}.

Finally, we close the paper by collecting the consequences of our results in terms of cardinal 
characteristics of the ideals, as well as formulating open questions related to the paper. 

 Regarding  more recent methods used in the field, a special role in our investigations is played 
 by the so-called  ``coherent omission of intervals'' invented by B.~Tsaban in \cite{OrTs11}, which led to the  groundbreaking construction of $\gamma$-sets from $\mathfrak p=\mathfrak b$,  in this context a rather modest set-theoretic assumption. This method might be thought of as a~state of the~art  application of the analysis of $\omega$-covers pioneered by F.~Galvin and A.~Miller in~\cite{GM}. 
\par

All the notions mentioned above that actually appear in some proofs (and not only serve as a motivation) are defined in the sections to
follow before their first occurrence.


\section{A~space of functions with the~Pytkeev property, but without the~Fr\' echet--Urysohn property}\label{S-nongamma_pytkeev}

Let us recall two well-known local topological properties. The~second one is named in~\cite{MaTi} after E.G.~Pytkeev who studied it in~\cite{Pytk84} as a~generalization of countable tightness.
\begin{dfn}
\begin{enumerate}[(1)]
    \item A~topological space~$Y$ has the~Fr\' echet--Urysohn property if for every $E\subseteq Y$ and $y\in\overline{E}$ there is a~sequence $\seqn{y_n}{n}$ in~$E$ converging to~$y$.
    \item A~topological space~$Y$ has the~Pytkeev property if for every $E\subseteq Y$ and $y\in\overline{E}\setminus E$ there is a~sequence $\seqn{E_n}{n}$ of infinite subsets of~$E$ such that each neighbourhood of~$y$ contains some~$E_n$.
\end{enumerate}
\end{dfn}
For more on the~Pytkeev property, see~\cite{Pytk84,MaTi,Sak03,SakQuad,SimTsa,TsZd09}. One can easily see that the~Pytkeev property follows from the~Fr\' echet--Urysohn property. In the~realm of arbitrary topological spaces, it is known that the~implication is strict, see~\cite{MaTi}. However, if we restrict our attention to the~space~$\cp{X}$, the~family of all continuous functions on a~topological space~$X$ equipped with the~topology of pointwise convergence, the~answer was not clear. In fact, it was raised as an~open problem first by M.~Sakai~\cite{Sak03}:
\begin{problem}[M.~Sakai]\label{prob-pytk_fu}
Does the~Pytkeev property imply the~Fr\' echet--Urysohn property for topological spaces of the~form~$\cp{X}$?
\end{problem}
In fact, Problem~\ref{prob-pytk_fu} was asked as Question~3 in~\cite{Sak03}, Problem~4.2 in~\cite{SakQuad}, and Problem~3.1 in~\cite{TsZd09}. Using a~set-theoretic assumption weaker than~$\CH$, we construct an~example answering the~above question negatively. The~definitions of the~cardinals~$\pp$ and~$\cc$ and related ones and their properties are available in~\cite{blass,buk-str}.
\begin{theorem}[$\pp=\cc$]\label{pi-nongamma}
There is a~set of reals~$A$ such that $\cp{A}$ has the~Pytkeev property, but not the~Fr\' echet--Urysohn property.
\end{theorem}
\par
The~proof of Theorem~\ref{pi-nongamma} applies some standard tools developed in the~field and related to topology, combinatorics, and open covers of a~set of reals~$A$. Let us begin with properties of~$A$ guaranteeing that $\cp{A}$ has the~Pytkeev property, and does not have the~Fr\' echet--Urysohn property.
\par
A~family~$\V$ of subsets of $X$ is called an~$\omega$-cover, if $X\not\in\V$ and for every finite $F\subseteq X$ there is $V\in\V$ such that $F\subseteq V$. An~infinite~$\V$ is called a~$\gamma$-cover, if $X\not\in\V$ and each $x\in X$ belongs to all but finitely many elements of~$\V$. Finally, a~topological space~$X$ is a~$\gamma$-set if every  open $\omega$-cover~$\V$ of~$X$ contains a~$\gamma$-subcover $\V'\subseteq\V$. J.~Gerlits and Zs.~Nagy~\cite{GerNag} have shown the~following.
\begin{theorem}[J.~Gerlits--Zs.~Nagy]\label{duality-gamma}
Let $X$ be a~Tychonoff space. $\cp{X}$ has the~Fr\' echet--Urysohn property if and only if $X$ is a~$\gamma$-set.
\end{theorem}
\par
We say that a~topological space~$X$ is a~$\pi$-set if $X$ is zero-dimensional and for every clopen $\omega$-cover~$\V$ of~$X$ there is a~family~$\set{\V_k}{k\in\omega}$ of infinite subfamilies of~$\V$ such that the~family~$\set{\bigcap\V_k}{k\in\omega}$ is an~$\omega$-cover of~$X$. For the~original formulation of a~$\pi$-set, see~\cite{Sak03}. Our formulation is by P.~Simon and B.~Tsaban~\cite{SimTsa}, and is equivalent to the~original one in a~Tychonoff topological space. M.~Sakai~\cite{Sak03} has shown the~following. 
\begin{theorem}[M.~Sakai]\label{duality-pi}
Let $X$ be a~Tychonoff space. $\cp{X}$ has the~Pytkeev property if and only if $X$ is a~$\pi$-set.
\end{theorem}
\par
All our (counter)examples are  subspaces of~$\PP(\omega)$, equipped with the~topology copied from the~Cantor space~${}^\omega2$ via the~bijection $\omega\supseteq a\mapsto\chi_a\in {}^\omega2$, where $\chi_a$ is the~characteristic function of~$a$. This way, the~family $\set{\set{a\subseteq\omega}{a\cap n=s}}{n\in\omega, s\subseteq n}$ is a~base of~$\PP(\omega)$. Thus, we can reduce our considerations just to countable covers. According to~\cite{GerNag}, a~topological space~$X$ has property~$\varepsilon$ if every  open $\omega$-cover of~$X$ contains a~countable $\omega$-subcover of~$X$. A~topological space~$Y$ has countable tightness if for every $E\subseteq Y$ and $y\in\overline{E}$ there is a~countable set~$F\subseteq E$ such that $y\in\overline{F}$. The~following characterization is a~combination of results obtained in~\cite{Arch76,GerNag,Pytk82}.
\begin{theorem}[A.V.~Arkhange\ml ski\v i, E.G.~Pytkeev, J.~Gerlits--Zs.~Nagy]\label{epsilon-char}
Let $X$ be a~Tychonoff space. The~following are equivalent.
\begin{enumerate}[(a)]
    \item  $X$ is an~$\varepsilon$-space.
    \item Every finite power of~$X$ is Lindel\"of.
    \item $\cp{X}$ has countable tightness.
\end{enumerate}
\end{theorem}
\par
As usually, we identify an~infinite set $a\in\infin$ with the~unique increasing enumeration of its elements. Thus a~set $a\in\infin$ is often treated as an~increasing function $a\in\baire\omega$, and $a(n)$ stands for the~$n$-th element of~$a$. This convention is useful since we do not have to introduce a~separate notation for the~family of all increasing sequences of natural numbers. 
\par
The~following lemma is a~key tool to deduce that the~set~$A$ constructed in~Theorem~\ref{pi-nongamma} is not a~$\gamma$-set. It is due to~\cite{GM}, but became a~folklore. A~family $A\subseteq\PP(\omega)$ is \tall\ if for any $b\in\infin$ there is $a\in A$ such that $a\cap b$ is infinite. According to~\cite{BDS}, we say that $A\subseteq\PP(\omega)$ has the~finite union property, shortly \fup, if for any $a_0,\dots,a_k\in A$, the~set $\omega\setminus\bigcup_{i=0}^{k}a_i$ is infinite.
\begin{lema}[folklore]\label{nothavinggamma}
Let $A\subseteq\PP(\omega)$ possess \fup. If $A$ is \tall, then $A$ is not a~$\gamma$-set. 
\end{lema}
\begin{proof}
Let us assume that $A$ has \fup. Then $\V=\set{V_n}{n\in\cup A}$ for $V_n=\set{a\in A}{n\not\in a}$ is an~$\omega$-cover of~$A$,
since no $V_n$ contains entire $A$. 
\par
To get the~contradiction, let us assume that $A$ is a~$\gamma$-set. There is $I\in\infin$ such that $\set{V_n}{n\in I}$ is a~$\gamma$-cover ($V_n$ mutually distinct for $n\in I$). Since $A$ is \tall\ there is $a\in A$ such that $a\cap I$ is infinite. However, for any $n\in a\cap I$ we have $a\not\in V_n$, a~contradiction.
\end{proof}
\par
Depending on the~context, $\fin$ is the~family of all finite subsets of~$\omega$, ${}^1\omega$, $\omega\times\omega$, etc. All counterexamples in this article are constructed using transfinite induction and have the~form 
\[
A=\set{a_\alpha}{\alpha<\cc}\cup\fin
\]
for suitably chosen $a_\alpha$'s so that $A$ possesses the~desired properties. This approach follows the~legacy of the~construction by F.~Galvin and A.W.~Miller~\cite{GM}, which was afterwards generalized by B.~Tsaban and collaborators, see, e.g., \cite{OrTs11,ST2, Ts11,TsZd08}. 
\par
If $n<m$, then the~family $\set{a\subseteq\omega}{[n,m)\cap a=\emptyset}$ is an~open set in~$\PP(\omega)$. Thus an~infinite set~$b\in\infin$ induces an~open $\gamma$-cover~$\seqn{V_j}{j}$ of~$\fin$, defined by $V_j=\set{a\subseteq\omega}{[b(j),b(j+1))\cap a=\emptyset}$. Due to following folklore lemma observed in~\cite{GM}, such open sets play a~crucial role in many constructions.
\begin{lema}[F.~Galvin--A.W.~Miller]\label{GMlemma}
Let $E\in\infin$.
\begin{enumerate}[(1)]
    \item If $\V$ is an~open $\omega$-cover of~$\fin$, then there are $b\in[E]^\omega$ and a~family $\set{V_j}{j\in\omega}\subseteq\V$ of distinct sets such that $a\in V_j$ if $a\cap[b(j),b(j+1))=\emptyset$.
    \item Let $A\supseteq\fin$ be of cardinality less than~$\pp$, $\V$ being an~open $\omega$-cover of~$A$. Then there is $b\in[E]^\omega$ and a~$\gamma$-cover $\set{V_j}{j\in\omega}\subseteq\V$ of $A$ such that
    \[
    V_j\supseteq\set{a\subseteq\omega}{[b(j),b(j+1))\cap a=\emptyset}.
    \]
\end{enumerate}
\end{lema}
\par
By an~ideal on $M$, we mean a~family $\I\subseteq \PP(M)$ that is
closed under taking subsets (we call such families \emph{hereditary}), closed under finite unions, contains all finite subsets of~$M$, and $M\notin\I$. We consider just ideals on countable sets. Calligraphic $\I$, $\J$, and $\K$ are used exclusively to denote ideals, most often on~$\omega$ or~$\omega\times\omega$. We are now ready to proceed to the~proof of Theorem~\ref{pi-nongamma}.
\par
\begin{proof}[Proof of Theorem~\ref{pi-nongamma}]
Let us fix an~auxiliary partition $\set{I_n}{n\in\omega}\subseteq\infin$ of~$\omega$, and a~\tall\ ideal~$\I$ on~$\omega$. 
\par
\begin{clma}\label{onestep-pi-nong}
For each $b,c\in\infin$ there is $a\in[c]^\omega $ such that  
    \[
    |\set{n\in\omega}{(\exists^\infty j\in I_n)\ [b(j),b(j+1))\cap a\neq\emptyset}| \leq 1.    
    \]
\end{clma}
\begin{proof}
Let $N=\set{n\in\omega}{(\exists^\infty j\in I_n)\ [b(j),b(j+1))\cap c\neq\emptyset}$. If $N=\emptyset$ there is nothing to prove. Otherwise pick $n\in N$ and set $a=c\cap\bigcup_{j\in I_n}[b(j), b(j+1)).$
\end{proof}
\par
Using Claim~\ref{onestep-pi-nong} we recursively build a~set $\set{a_\alpha}{\alpha<\cc}$ with strong combinatorial properties. Let us fix an~auxiliary enumeration $\set{b_\alpha}{\alpha<\cc}$ of all sets from~$\infin$ (each set repeated $\cc$ times).  
\begin{clma}\label{existence-pi-nong}
There is a~\tall\ set $\set{a_\alpha}{\alpha<\cc}\subseteq\infin\cap\I$ such that 
    \begin{center}
    $|\set{n\in\omega}{(\exists^\infty j\in I_n)\ [b_\beta(j),b_\beta(j+1))\cap a_\alpha\neq\emptyset}|\leq1$ for each $\beta\leq\alpha$.    
    \end{center}
\end{clma}
\begin{proof}
We fix an~enumeration $\set{c_\alpha}{\alpha<\cc}$ of all infinite subsets of~$\omega$. Assuming that we have $\set{a_\beta}{\beta<\alpha}$, we shall pick~$a_\alpha\in\infin$ such that the~conditions (i) - (iii) are satisfied. 
\par
 \begin{enumerate}[(i)]
    \item $a_\alpha\subseteq c_\alpha$,
    \item $a_\alpha\in\I$,
    \item $|\set{n\in\omega}{(\exists^\infty j\in I_n)\ [b_\beta(j),b_\beta(j+1))\cap a_\alpha\neq\emptyset}|\leq1$ for any $\beta\leq\alpha$.   
\end{enumerate}
\par
To find such $a_\alpha$, we shall first construct  a~$\subseteq^\ast$-decreasing sequence $\gseqn{a_{\alpha,\beta}}{\beta\leq\alpha}$ such that $c_\alpha\supseteq a_{\alpha,\beta}$, $a_{\alpha,\beta}\in\I$, and 
\[
|\set{n\in\omega}{(\exists^\infty j\in I_n)\ [b_\beta(j),b_\beta(j+1))\cap a_{\alpha,\beta}\neq\emptyset}|\leq1.
\]
For $\beta=\xi+1$ it is enough to directly apply Claim~\ref{onestep-pi-nong} for $a_{\alpha,\xi}$. If $\beta$ is limit, then we select a~pseudointersection~$a\subseteq c_\alpha$ of $\gseqn{a_{\alpha,\xi}}{\xi\leq\beta}$, and afterwards we apply Claim~\ref{onestep-pi-nong} for a~pseudointersection~$a$. Once we have $\gseqn{a_{\alpha,\beta}}{\beta\leq\alpha}$, we pick arbitrary pseudointersection $a_\alpha\subseteq c_\alpha$ of $\gseqn{a_{\alpha,\beta}}{\beta\leq\alpha}$.
\end{proof}
\par
The~set $A=\set{a_\alpha}{\alpha<\cc}\cup\fin$, constructed in Claim~\ref{existence-pi-nong}, is \tall\ and possesses \fup. Therefore by Lemma~\ref{nothavinggamma}, $A$ is not a~$\gamma$-set. \par
Let $\V_0$ be an~$\omega$-cover of~$A=\set{a_\alpha}{\alpha<\cc}\cup\fin$, so an~$\omega$-cover of~$\fin$. By Lemma~\ref{GMlemma}(2), there is $\alpha_0<\cc$ and a~family $\set{V_{0,j}}{j\in\omega}\subseteq\V_0$ of distinct sets such that 
\[
V_{0,j}\supseteq\set{a\subseteq\omega}{[b_{\alpha_0}(j),b_{\alpha_0}(j+1))\cap a=\emptyset}.
\]
Let us suppose we have $\V_0,\dots,\V_m$. $\V_0$ is an~$\omega$-cover of~$\set{a_\alpha}{\alpha<\alpha_m}\cup\fin$, so there is a~$\gamma$-cover~$\V_{m+1}\subseteq\V_0$ covering~$\set{a_\alpha}{\alpha<\alpha_m}\cup\fin$. Anyway, $\V_{m+1}$ is an~$\omega$-cover of~$\fin$. Again, by Lemma~\ref{GMlemma}(2) there is an~increasing sequence $b_{\alpha_{m+1}}$ with $\alpha_{m+1}>\alpha_m$ and a~family $\set{V_{m+1,j}}{j\in\omega}\subseteq\V_{m+1}$ of distinct sets such that 
\[
V_{m+1,j}\supseteq\set{a\subseteq\omega}{[b_{\alpha_{m+1}}(j),b_{\alpha_{m+1}}(j+1))\cap a=\emptyset}.
\]
\par
We set $\alpha_\ast=\sup\set{\alpha_i}{i\in\omega}$ and $\W_{i,n}=\set{V_{i,j}}{j\in I_n}$. 
\par
We claim that for any $F\in[A]^{<\omega}$ there are $i,n\in\omega$ such that $F\subseteq V$ for all but finitely many $V\in\W_{i,n}$. Indeed, pick an~arbitrary $F\in[A]^{<\omega}$, and set $J=\set{\alpha}{a_\alpha\in F}$. Also, let $J_<=J\cap\alpha_\ast$ and $J_\geq=J\cap[\alpha_\ast,\cc)$. 
\par
If $\alpha\in J_\geq$, then $\alpha\geq\alpha_i$ for each $i\in\omega$. By the~assumption, there is $n_i\in\omega$ such that 
\[
(\forall\alpha\in J_\geq)(\forall^\infty j\in I_{n_i})\ [b_{\alpha_i}(j),b_{\alpha_i}(j+1))\cap a_\alpha=\emptyset.
\]
Hence, $\set{a_\alpha}{\alpha\in J_\geq}\subseteq V_{i,j}$ for all but finitely many $j\in I_{n_i}$. One can see that $\set{a_\alpha}{\alpha\in J_\geq}\cup(F\cap\fin)\subseteq V_{i,j}$ for all but finitely many $j\in I_{n_i}$ as well. For all $\alpha\in J_<$ we have $\alpha<\alpha_m$ for some $m\in\omega$. The~family $\set{V_{i,j}}{j\in I_{n_i}}$ is a~$\gamma$-cover of $\set{a_\alpha}{\alpha<\alpha_i}\cup\fin$ for each $i>m$. Hence, $F\subseteq V$ for all but finitely many $V\in\W_{m+1,n_{m+1}}$.
\end{proof}

\section{A~solution of the~Gerlits--Nagy problem: a~$\delta$-set that is not a~$\gamma$-set}\label{S-nongamma_delta}

In the~seminal paper by~J.~Gerlits and Zs.~Nagy~\cite{GerNag}, they studied several local topological properties of~$\cp{X}$, including the~Fr\' echet--Urysohn property, and found the~corresponding topological property of the~domain space~$X$. They discovered a~$\delta$-set, the~notion having close connection to the~latter properties. Let us recall its definition. 
\par
According to~\cite{GerNag}, for a~family $\set{V_n}{n\in\omega}\subseteq\PP(X)$, 
\[
\Liminf\set{V_n}{n\in\omega}=\bigcup\limits_{n\in\omega}\bigcap\limits_{m\geq n}V_m=\set{x\in X}{(\forall^\infty n\in\omega)\ x\in V_n}.
\]
For $\V\subseteq\PP(X)$, the~family~$\rmL(\V)$ is the~smallest family containing~$\V$ and closed under the~operator~$\Liminf$. A~topological space~$X$ has property~$\delta$ (alternatively, is a~$\delta$-space, a~$\delta$-set) if for every  open $\omega$-cover~$\V$ of~$X$ we have $X\in\rmL(\V)$.
\par
Right from the~definition one can see that any~$\gamma$-set is a~$\delta$-set. Moreover, it has been shown in~\cite{GerNag} that any $\delta$-set of reals has strong measure zero. Thus it is consistent that both properties are equivalent. Therefore, J.~Gerlits and Zs.~Nagy~\cite{GerNag} raised as Problem below their Theorem~7 the~following question.
\begin{problem}[J.~Gerlits--Zs.~Nagy]\label{prob-delta_gamma}
Is every $\delta$-set also a~$\gamma$-set?
\end{problem}
\par
The~problem was mentioned by A.~Miller~\cite{Mi05} during his plenary lecture at the~Second Workshop on Coverings, Selections, and Games in Topology, Lecce, Italy. T.~Orenshtein and B.~Tsaban~\cite{OrTs13} restate the~question as their Problem~1.5, and mention it is the~only still open problem posed in~\cite{GerNag}. We shall prove the~following.\footnote{During the preparation of this manuscript we have been informed by Jialiang He that he has also solved this problem under~$\CH$.}
\begin{theorem}[$\pp=\cc$]\label{delta-nongamma}
There is a~$\delta$-set of reals $A\subseteq\PP(\omega)$ that is not a~$\gamma$-set.
\end{theorem}
\begin{proof}
We follow the~proof of Theorem~\ref{pi-nongamma}. Therefore, let us fix an~auxiliary partition $\set{I_n}{n\in\omega}\subseteq\infin$ of~$\omega$, and construct the~set $A=\set{a_\alpha}{\alpha<\cc}\cup\fin$ as in the~proof of Theorem~\ref{pi-nongamma}. Thus, $A$ is not a~$\gamma$-set. The~proof of the~fact that $A$ has property~$\delta$ is the~same as the~proof of the~fact that $A$ has property~$\pi$ in Theorem~\ref{pi-nongamma} until one defines $\alpha_\ast=\sup\set{\alpha_i}{i\in\omega}$. Afterward, taking $a\in A$, we shall show that 
\[
a\in\Liminf\set{\Liminf\big\{\Liminf\{V_{i,j}:j\in I_n \}: n\in\omega \big\}}{i\in\omega},
\]
and so $A\in L(\mathcal V_0)$, which would complete our proof. Indeed, first, we take $a=a_\alpha$ for $\alpha\geq\alpha_\ast$. Then we have
\[
(\forall i)(\forall^\infty n)(\forall^\infty j\in I_n)\ [b_{\alpha_i}(j),b_{\alpha_i}(j+1))\cap a_{\alpha}=\emptyset.
\]
Hence, using the~definition of~$V_{i,j}$'s, the~latter equality yields $a_\alpha\in V_{i,j}$. For the~second case, it remains to take $a\in\set{a_\alpha}{\alpha<\alpha_m}\cup\fin$ for some $m$. However, the~construction of the~family of~$V_{i,j}$'s yields that the~family $\set{V_{i,j}}{j\in\omega}$ is a~$\gamma$-cover of $\set{a_\alpha}{\alpha<\alpha_m}\cup\fin$ for each $i>m$. 
\end{proof}
\par

By a~result of M.~Sakai~\cite{SakQuad}, any $\delta$-set is a~$\pi$-set. Therefore, the~set constructed in Theorem~\ref{delta-nongamma} is a~$\pi$-set. Hence, Theorem~\ref{pi-nongamma} is a~consequence of Theorem~\ref{delta-nongamma}.


\section{Kat\v etov power ideals}\label{S-katetov}

\par
We include a~technical section to introduce Kat\v etov power ideals. However, those familiar with terminology on ideals and with Kat\v etov power ideals may skip reading the~section. We recalled the~definition of ideal in Section~\ref{S-nongamma_pytkeev}, although our terminology on ideals is in accordance with standard sources~\cite{BrFl17,BrFaVe,Hr,Hr17}.
\par
Recall from \cite{kat68,kat72} that for a~family $\set{\I_a}{a\in N}$, each $\I_a$ being an~ideal on $N_a$, and an~ideal $\I$ on~$N$, the~family 
\[
\I\text{-}\sum\set{\I_a}{a\in N}=\left\{\bigcup_{a\in N}(\{a\}\times I_a)\colon(\exists I\in\I)(\forall a\in N\setminus I)\ I_a\in\I_a\right\}
\]
is an~ideal on~$\sum_{a\in N}N_a=\set{(a,x)}{x\in N_a, a\in N}$. In particular, $\K_1\times\K_2=\K_1\text{-}\sum\set{\K_n}{n\in N}$, where $N$ is the~domain of the~ideal~$\K_1$ and $\K_n=\K_2$ for any $n\in N$. Moreover, 
\[
\I_0\oplus\dots\oplus\I_{m-1}=\emptyset_m\text{ -}\sum\set{\I_i}{i\in m},
\]
where $\emptyset_m=\{\emptyset\}$. Note that $m$ does not carry any ideal according to our definition, but $\emptyset_m$ is hereditary, closed under finite unions, and $m\notin\emptyset_m$. Thus, $\I_0\oplus\dots\oplus\I_{m-1}$ is an~ideal on~$\sum_{i\in m}N_i$. For a~sequence of ideals $\seqn{\J_n}{n}$ we denote $\J_\infty=\fin\text{-}\sum\set{\J_n}{n\in\omega}$.
\par
We recall the~definition of the~ideal $\fin^\alpha$ introduced in~\cite{kat72,Gri}, recently studied in~\cite{BFMS,DeRa09,KwRe13}, and broadly applied in~\cite{BaScTo,Do16,JsKelo02,RaSt}. Throughout the~paper, the~ideal $\K_\alpha$ is the~ideal on limit~$\alpha$ consisting of all subsets of ordinals smaller than~$\alpha$. $\fin^1$ is the~ideal $\fin$. For $1<\alpha<\omega_1$, we set
\[
\fin^\alpha=
\begin{cases}
\fin\times\fin^\beta&\alpha=\beta+1,\\
\K_\alpha\text{-}\sum\set{\fin^\beta}{\beta<\alpha}&\alpha\text{ limit}.
\end{cases}
\]
The~domain~$D_\alpha$ is defined inductively by $D_\alpha=\omega\times D_\beta$ for $\alpha=\beta+1$, and $D_\alpha=\sum_{\beta<\alpha} D_\beta$ for $\alpha$ limit.
\par
In the~proof of the~main result of  the~next section, in~Lemma~\ref{Lalpha-finalpha}, we use a~variant of original Kat\v etov power ideal~$\fin^\alpha$, which we denote by $\fin^\alpha_M$. However, we show that both ideals are K-equivalent. Let us now introduce $\fin^\alpha_M$ together with the~auxiliary terminology.
\par
We say that a~set $M\subseteq\baire\omega_1$ {\it contains cofinal sequences} if the~following conditions are satisfied.
\begin{enumerate}[(1)]
    \item Each $\xi\in M$ is strictly increasing or constant.
    \item For each limit $\alpha<\omega_1$ there is a~unique $\xi\in M$ such that $\set{\xi(n)}{n\in\omega}$ is cofinal in~$\alpha$.
    \item $M$ contains all constant sequences.
\end{enumerate}
\par
We fix a~set~$M$ containing cofinal sequences, and we denote $\fin^1_M$ to be~an ideal of all finite subsets of~${}^1\omega$. For $1<\alpha<\omega_1$ and $\seqn{\alpha_n}{n}\in M$ cofinal in~$\alpha$, we set
\[
\fin^\alpha_M=\fin\text{-}\sum\set{\fin^{\alpha_n}_M}{n\in\omega}.
\]
In~particular, for $\alpha=\beta+1$ we have $\fin^\alpha_M=\fin\times\fin^\beta_M$. $\fin^\alpha_M$ is an~ideal on a~set~$D^\alpha_M$ that may be defined by induction as follows: $D^\alpha_M=\sum_{n\in\omega}D^{\alpha_n}_M$ for $\alpha<\omega_1$. In~particular, $D^\alpha_M=\omega\times D^\beta_M$ for $\alpha=\beta+1$. One can see that the~ideal $\fin^n$ coincides with ideal $\fin^n_M$ for $n\in\omega$. For simplicity we often treat $D^\alpha_M$ as a~set of sequences, i.e., $D^\alpha_M=\set{n^\smallfrown t}{t\in D^{\alpha_n}_M,n\in\omega}$, where by $n^\smallfrown t$, we mean the~concatenation~$\langle n\rangle^\smallfrown t$ of the~one-element sequence~$\langle n\rangle$ with the~sequence~$t$. Note that $D^n_M$ for $n\in\omega$ is the~set of all sequences of length~$n$ or the~set of all $n$-tuples. 
\par
We list the~basic properties of ideal~$\fin^\alpha_M$ in Proposition~\ref{finbetaleqfinalpha}. However, we need to recall basic terminology for comparing ideals. For $\square=\text{\rm 1-1}, \mathrm{KB}, \mathrm{K}$ we say that $\varphi\in\baire\omega$ is $\square$-function if $\varphi$ is one-to-one, finite-to-one, arbitrary, respectively. For $\A_1,\A_2\subseteq\PP(\omega)$ and $\varphi\in\baire\omega$, we write $\A_1\leq_\varphi\A_2$ if $\varphi^{-1}(K)\in\A_2$ for any $K\in\A_1$, and $\A_1\leq_\square\A_2$ if there is a~$\square$-function $\varphi$ such that $\A_1\leq_\varphi\A_2$ (Similarly for $\varphi\in{}^{M_2}M_1$ and $\A_1\subseteq\PP(M_1)$, $\A_2\subseteq\PP(M_2)$). One may easily see that 
\[
\A_1\subseteq\A_2\Rightarrow\A_1\leqoo\A_2\Rightarrow\A_1\leqkb\A_2\Rightarrow\A_1\leqk\A_2. 
\]
If $\A_1\supseteq\fin$ and $\A_1\leq_\varphi\A_2$, then $\varphi$ is $\A_2$-to-one. For general products, $\K_1\leqk\K_1\times\K_2$ and $\K_2\leqkb\K_1\times\K_2$, and for sums, $\I_0\oplus\dots\oplus\I_n\leqoo\I_i$ for each $0\leq i\leq n$. For $\square=\text{\rm 1-1}, \mathrm{KB}, \mathrm{K}$, we say that $\A_1,\A_2$ are $\square$-equivalent if $\A_1\leq_\square\A_2$ and $\A_2\leq_\square\A_1$.
\par
Note that $\finfin\leqk\J$ if and only if $\finfin\leqoo\J$, see~\cite{FiSz}. For other equivalent conditions to $\finfin\leqk\J$ and further references, see~\cite{FiSz,Su16,KwSt17}. Similarly, $\fin^\alpha\leqk\J$ if and only if $\fin^\alpha\leqoo\J$ by~\cite{BFMS}. It is known that $\fin^\alpha\leqoo\fin^\beta$ and $\fin^\beta\not\leqk\fin^\alpha$ for $\alpha<\beta$, see~\cite{DeRa09} or our Proposition~\ref{finbetaleqfinalpha}. At first reading, one can skip proofs of Proposition~\ref{finbetaleqfinalpha} because they are intuitively clear, but technical.
\begin{prop}\label{finbetaleqfinalpha}
Let $0<\alpha<\omega_1$ and $M, N\subseteq\baire\omega_1$ contain cofinal sequences. 
\begin{enumerate}[(1)]
    \item If $0<\beta<\alpha$, then $\fin^\beta_M\leqoo\fin^\alpha_M$.
    \item $\fin^\alpha_M$ and $\fin^\alpha_N$ are {\rm ($\oo$)}-equivalent.
    \item $\fin^\alpha_M$ and $\fin^\alpha_M\oplus\dots\oplus\fin^\alpha_M$ are {\rm ($\oo$)}-equivalent.
    \item $\fin^\alpha_M$ is {\rm ($\oo$)}-equivalent to $\fin^\alpha_M\restriction Y$ for any $Y\not\in\fin^\alpha_M$.
    \item $\fin^\alpha_M$ is {\rm K}-equivalent to $\fin^\alpha$. Moreover, $\fin^\alpha\leqoo\fin^\alpha_M$.
    \item $\fin^\alpha_M$ is not {\rm K}-equivalent to $\fin^\beta_M$ for $\alpha\neq\beta$.
\end{enumerate}
\end{prop}
\par
We need an~auxiliary assertion.
\begin{lema}\label{finalphainductivestep}
Let $\omega\leq\alpha<\omega_1$, and $\set{\I_\beta}{\beta<\alpha}$, $\set{\J_\beta}{\beta<\alpha}$ be families of ideals on countable sets.
\begin{enumerate}[(1)]
    \item If $\I_k\leqsq\J_k$  for all but finitely many $k\in\omega$, then $\I_\infty\leqsq\J_\infty$.
    \item Let $\seqn{F_k}{k}$ be a~decomposition of~$\omega$ into finite sets (including possibly empty set). If $\I_k\leqoo\I_k\oplus\dots\oplus\I_k$ ($i$ times) for all $i\in\omega$, and $\I_k\leqoo\J_m$  for $m\in F_k$ and all but finitely many $k\in\omega$, then $\I_\infty\leqoo\J_\infty$.
\end{enumerate}    
If $\seqn{\alpha_n}{n}\in M$ is cofinal in~$\alpha$, then we have the~following.
\begin{enumerate}
    \item[\rm (3)]  If $\I_{\alpha_k}\leqsq\J_k$  for all but finitely many $k\in\omega$, then $\K_\alpha\text{-}\sum\set{\I_\beta}{\beta<\alpha}\leqsq\J_\infty$.
    \item[\rm (4)] If $\I_k\leqk\J_\beta$ for $\alpha_k\leq\beta<\alpha_{k+1}$ and all but finitely many~$k\in\omega$, then $\I_\infty\leqk\K_\alpha\text{-}\sum\set{\J_\beta}{\beta<\alpha}$.
\end{enumerate}
\end{lema}
\begin{proof}
For a~subset $A\subseteq\alpha\times\omega$, and $\beta<\alpha$, let $A_\beta$ be the~projection of $\beta$-th vertical section of~$A$, i.e., $A_\beta=\set{n}{(\beta,n)\in A}$.
\par
(1) There are $k_0$ and $\square$-functions $\varphi_k\colon\omega\to\omega$ for $k\geq k_0$ such that $\I_k\leq_{\varphi_k}\J_k$. For $k<k_0$ we set $\varphi_k$ to be the~identity. We define $\varphi\colon\omega\times\omega\to\omega\times\omega$ by setting $\varphi(k,n)=(k,\varphi_k(n))$. We claim that $\I_\infty\leq_\varphi\J_\infty$. Indeed, let $A\in\I_\infty$, and $k\geq k_0$. Then $A_k\in\I_k$, and therefore $\varphi_k^{-1}[A_k]\in\J_k$. However, using the~definition of~$\varphi$ one can easily check that $(\varphi^{-1}[A])_k=\varphi_k^{-1}[A_k]$.
\par
(2) Let $\I_k'=\I_k\oplus\dots\oplus\I_k$ ($|F_k|$ times), and $\J_k'=\emptyset_{F_k}\text{ -}\sum\set{\J_m}{m\in F_k}$. By respective assumptions, we obtain $\I_k\leqoo\I_k'\leqoo\J_k'$ that by~(1) leads to $\I_\infty\leqoo\J_\infty'$. However, one can easily see that $\J_\infty'$ and $\J_\infty$ are ($\oo$)-equivalent.
\par
(3) By the~assumption and part~(1), we immediately obtain $\fin\text{-}\sum\set{\I_{\alpha_k}}{k\in\omega}\leqsq\J_\infty$. We shall show that 
\[
\K_\alpha\text{-}\sum\set{\I_\beta}{\beta<\alpha}\leq_\varphi\fin\text{-}\sum\set{\I_{\alpha_k}}{k\in\omega}
\]
for an~injective $\varphi\colon\omega\times\omega\to\alpha\times\omega$ defined by $\varphi(k,n)=(\alpha_k,n)$. Indeed, let $A\in\K_\alpha\text{-}\sum\set{\I_\beta}{\beta<\alpha}$. Then there is $\beta_0<\alpha$ such that $A_\beta\in\I_\beta$ for all $\beta>\beta_0$. Taking $k_0\in\omega$ with $\alpha_{k_0}>\beta_0$ we obtain $A_{\alpha_k}\in\I_{\alpha_k}$ for all $k>k_0$.
\par
(4) Similarly to previous cases, constructing an~appropriate map using the~assumption we obtain 
\[
\I_k\leqk\J_k'=\emptyset_{[\alpha_k,\alpha_{k+1})}\text{-}\sum\set{\J_\beta}{\alpha_k\leq\beta<\alpha_{k+1}}
\]
for all but finitely many $k$'s, where $\emptyset_{[\alpha_k,\alpha_{k+1})}$ is the~family containing just empty set. By part~(1), we then have $\I_\infty\leqk\J_\infty'$. However, one can check that ideals $\J_\infty'$ and $\K_\alpha\text{-}\sum\set{\J_\beta}{\beta<\alpha}$ are ($\oo$)-equivalent.
\end{proof}
\par
We are ready to proceed to the~proof of Proposition~\ref{finbetaleqfinalpha}.
\begin{proof}[Proof of Proposition~\ref{finbetaleqfinalpha}]
(1) Let $\seqn{\xi_n}{n}\in M$ be cofinal in~$\xi$. We shall show by transfinite induction that $\fin^{\xi_n}_M\leqoo\fin^\xi_M$ for any $1<\xi<\omega_1$. Indeed, let $\seqn{\xi_{n,k}}{k}\in M$ be cofinal in~$\xi_n$. By the~inductive assumption, we have $\fin^{\xi_{n,k}}_M\leqoo\fin^{\xi_k}_M$ for each $k\geq n$. By Lemma~\ref{finalphainductivestep}(1), we obtain the~conclusion.
\par
(2) - (3) We shall show by transfinite induction, that $\fin^\alpha_M\leqoo\fin^\alpha_N$ and $\fin^\alpha_M\leqoo\fin^\alpha_M\oplus\dots\oplus\fin^\alpha_M$ ($i$ times) for any $1<\alpha<\omega_1$ and $i\in\omega$. Thus we assume that it holds for all $\beta<\alpha$. Moreover, one can easily see that $\fin^n_M=\fin^n_N$ for any $n\in\omega$.
\par
If $\alpha=\beta+1$, then by inductive assumption, we have $\fin^\beta_M\leqoo\fin^\beta_N$. By Lemma~\ref{finalphainductivestep}(1), we obtain $\fin^\alpha_M\leqoo\fin^\alpha_N$. 
\par
Let $\alpha$ be limit, $\seqn{\xi_n}{n}\in M$ and $\seqn{\eta_n}{n}\in N$ being cofinal in~$\alpha$. We set $N_n=[\xi_n,\xi_{n+1})\cap\set{\eta_m}{m\in\omega}\in\fin$. If $N_n$ is non-empty, then by inductive assumption, we have $\fin^{\xi_n}_M\leqoo\fin^{\xi_n}_M\oplus\dots\oplus\fin^{\xi_n}_M$ ($|N_n|$ times). If $\eta_k\in N_n$, then by inductive assumption and part~(1), we have $\fin^{\xi_n}_M\leqoo\fin^{\xi_n}_N\leqoo\fin^{\eta_k}_N$. By Lemma~\ref{finalphainductivestep}(2), we obtain $\fin^\alpha_M\leqoo\fin^\alpha_N$. 
\par
Finally, to obtain $\fin^\alpha_M\leqoo\fin^\alpha_M\oplus\dots\oplus\fin^\alpha_M$ ($m$ times), let $\seqn{\xi_n}{n}\in M$ be cofinal in~$\alpha$, and $\set{b_i}{i\in m}\subseteq\infin$ a~decomposition of~$\omega$ into infinite sets. Let $M_i\subseteq\baire\omega_1$ be same as~$M$, except for the~cofinal sequence in~$\alpha$ is now $\seqn{\xi_{b_i(n)}}{n}$, where $b_i(n)$ is the~$n$-th element of~$b_i$ in its increasing enumeration. It follows from the~above that $\fin^\alpha_{M_i}\leqoo\fin^\alpha_M$ for all~$i$, and hence we are left with the~task to show that $\fin^\alpha_M\leqoo\fin^\alpha_{M_0}\oplus\dots\oplus\fin^\alpha_{M_{m-1}}$. This is witnessed by the~bijection $\langle i, n^\smallfrown t\rangle\mapsto b_i(n)^\smallfrown t$ between their domains, where $i\in m$, $n\in\omega$, and $t\in D^{\xi_{b_i(n)}}_{M_i}=D^{\xi_{b_i(n)}}_M$. 
\par 
(4) One can easily check that the~statement is true for ideal $\fin_M^2$. Let us continue using transfinite induction, so let us assume that the~statement is true for all $\beta<\alpha$, and fix a~sequence $\seqn{\alpha_n}{n}\in M$, cofinal in~$\alpha$. We assume that $Y\not\in\fin^\alpha_M$. We denote $Y_n=\set{t\in D^{\alpha_n}_M}{n^\smallfrown t\in Y}$, and we set $I=\set{n}{Y_n\not\in\fin^{\alpha_n}_M}$. By the~choice of~$Y$, we obtain that $I$ is infinite. Furthermore, by the~inductive hypothesis, the~ideal $\fin^{\alpha_n}_M\restriction Y_n$ is ($\oo$)-equivalent to $\fin^{\alpha_n}_M$ for each $n\in I$. Thus by~Lemma~\ref{finalphainductivestep}(1), ideals $\I, \J$ are ($\oo$)-equivalent, where 
\begin{center}
$\I=\fin\text{-}\sum\set{\fin^{\alpha_n}_M\restriction Y_n}{n\in I}$, and $\J=\fin\text{-}\sum\set{\fin^{\alpha_n}_M}{n\in I}$.     
\end{center}
However, by part~(2), the~ideal~$\J$ is ($\oo$)-equivalent to~$\fin^\alpha_M$. On the~other hand, we shall show that $\I$ is ($\oo$)-equivalent to $\fin^\alpha_M\restriction Y$ that would finish the~proof. Indeed, the~identity map is a~witness for $\fin^\alpha_M\restriction Y\leqoo\I$. Finally, a~witness for $\I\leqoo\fin^\alpha_M\restriction Y$ is a~map~$\varphi$ such that $\varphi$ maps the~sets $\set{n^\smallfrown t}{t\in Y_n, n\not\in I}$ and $\set{I(0)^\smallfrown t}{t\in Y_{I(0)}}$ bijectively onto the~latter one ($I(0)$ being the~$0$-th element of~$I$), and on the~remaining part of~$Y$, $\varphi$ coincides with the~identity. 
\par
(5) Ideal $\fin^n$ coincides with the~ideal $\fin^n_M$ for $n\in\omega$. Let us assume that $\fin^\beta\leqoo\fin^\beta_M$ and $\fin^\beta_M\leqk\fin^\beta$ for each~$\beta<\alpha$. The~successor stage follows by~Lemma~\ref{finalphainductivestep}(1). Therefore we assume that $\alpha$ is limit. Then $\fin^\alpha\leqoo\fin^\alpha_M$ by~Lemma~\ref{finalphainductivestep}(3).  
\par
Let $\seqn{\alpha_n}{n}\in M$ be cofinal in~$\alpha$. By part~(1), for $\alpha_k\leq\beta<\alpha_{k+1}$ we have $\fin^{\alpha_k}_M\leqk\fin^\beta_M$. Together with inductive assumption that leads to $\fin^{\alpha_k}_M\leqk\fin^\beta$ for $\alpha_k\leq\beta<\alpha_{k+1}$. The~statement for $\alpha$ then follows by Lemma~\ref{finalphainductivestep}(4). 
\par
(6) $\fin^\alpha\not\leqk\fin$ since one may easily see that $\finfin\not\leqk\fin$, and it has been already shown that $\finfin\leqoo\fin^\alpha$ for $\alpha\geq 2$. 
\par
Let $\beta<\alpha$. By part~(1), we have $\fin^\beta_M\leqk\fin^\alpha_M$, and by part~(5), it is enough to show that $\fin^\alpha\not\leqk\fin^\beta$. However, this follows by results by M.~Kat\v etov~\cite{kat72}, see~\cite{DeRa09,FiSz} as well: He has shown in his Theorem~7.1 that considering all pointwise limits of all continuous functions on a~topological space with respect to~$\fin^\xi$, one obtains exactly $\xi$-th Baire class of functions.  If $\fin^\alpha\leqk\fin^\beta$, then by M.~Kat\v etov~\cite{kat72}, Proposition~4.2, the~$\alpha$-th Baire class of functions is below $\beta$-th Baire class of functions. However, the~$\beta$-th Baire class of functions on real line is proper subset of the~$\alpha$-th Baire class of functions, see e.g. \cite[Exercise~6.1]{buk-str}.
\end{proof}

\section{The~hierarchy of $\delta$-sets}\label{S-delta_hier}

\par
The~section is devoted to the~introduction of a~hierarchy of subsets of a~topological space that can be used to define a~hierarchy of $\delta$-sets. Main result of the~section is Lemma~\ref{Lalpha-finalpha} showing that this hierarchy can be described via $\fin^\alpha$-$\gamma$-covers for Kat\v etov power ideals~$\fin^\alpha$.  
\par
If $\V$ is arbitrary family of subsets of~$X$, using operator $\Liminf$, whose definition is in Section~\ref{S-nongamma_delta}, we may define the~following hierarchy:
\begin{center}
\begin{tabular}{ccl}
$\rmL^0(\V)$&$=$&$\V$,\\[0.2cm]
$\rmL^\xi(\V)$&$=$&$\set{\Liminf\set{W_n}{n\in\omega}}{(\forall n\in\omega)\ W_n\in\bigcup\limits_{\eta<\xi}\rmL^\eta(\V)}$.\\
\end{tabular}
\end{center}
One can see that $\rmL^\alpha(\V)\subseteq\rmL^\beta(\V)$ for $\alpha<\beta$, and $\rmL^\alpha(\V)=\rmL^{\omega_1}(\V)$ for $\alpha>\omega_1$. Hence, $\rmL(\V)=\rmL^{\omega_1}(\V)$. 
\par
We shall show that there is a~different way to represent the~latter hierarchy in~$\varepsilon$-spaces. Let us recall that for an~ideal~$\I$ on a~set~$M$, we say that a~sequence $\gseqn{V_m}{m\in M}$ of proper subsets of~$X$  is an~$\I$-$\gamma$-\emph{cover} of~$X$ if $\set{m}{x\not\in V_m}\in\I$ for every $x\in X$, see~\cite{BDS,dascomb,ST2}. One can check that for ideals $\I_1\leq_\varphi\I_2$ on~$M_1,M_2$, respectively, if $\gseqn{V_m}{m\in M_1}$ is an~$\I_1$-$\gamma$-cover of~$X$, then $\gseqn{V_{\varphi(m)}}{m\in M_2}$ is an~$\I_2$-$\gamma$-cover of~$X$, see Lemma~6.1(4), in~\cite{SoSu} for instance. Finally, let us recall from previous section that $D_\alpha$ is defined inductively by $D_\alpha=\omega\times D_\beta$ for $\alpha=\beta+1$, and $D_\alpha=\sum_{\beta<\alpha} D_\beta$ for $\alpha$ limit.
\par
\begin{lema}\label{Lalpha-finalpha}
Let $X$ be a~set, $G\subseteq X$, and $\V=\set{V_n}{n\in\omega}$ a~family of subsets of~$X$ such that $V_n\neq X$ for all $n\in\omega$. There is $G'\in\rmL^\alpha(\V)$ with $G'\supseteq G$ if and only if there is $\varphi\colon D^\alpha\to\omega$ such that $\gseqn{V_{\varphi(s)}}{s\in D^\alpha}$ is a~$\fin^\alpha$-$\gamma$-cover of~$G$.
\end{lema}
\begin{proof}
We take a~set~$M$ containing cofinal sequences, and we shall work with ideals $\fin^\alpha_M$, which are K-equivalent to~$\fin^\alpha$. The~proof is accomplished by induction. Let $\alpha\in\omega_1$ and let $\seqn{\alpha_n}{n}\in M$ be cofinal in~$\alpha$ (this includes $\alpha_n=\beta$ for all~$n$ if $\alpha=\beta+1$).
\par
Let us assume that $G'\in\rmL^\alpha(\V)$ with $G'\supseteq G$. By definition, there is $W_i\subseteq X$ and $\beta_i<\alpha$ such that $W_i\in\rmL^{\beta_i}(\V)$ and $G'=\Liminf\set{W_i}{i\in\omega}$. Since $\rmL^\beta(\V)$ is monotone with respect to~$\beta$, and $\seqn{\alpha_n}{n}$ is cofinal in~$\alpha$, there is an~increasing sequence~$\seqn{n_i}{i}$ such that $W_i\in\rmL^{\alpha_{n_i}}(\V)$. We define a~sequence $\seqn{U_n}{n}$ such that $G'=\Liminf\set{U_n}{n\in\omega}$ and $U_n\in\rmL^{\alpha_n}(\V)$ by setting $U_n=W_i$ for $n_i\leq n<n_{i+1}$ and $U_n=W_0$ for $n<n_0$. By the~inductive hypothesis on~$\alpha_n$, there is $\varphi_n\colon D^{\alpha_n}_M\to\omega$ such that $\gseqn{V_{\varphi_n(t)}}{t\in D_M^{\alpha_n}}$ is a~$\fin^{\alpha_n}_M$-$\gamma$-cover of~$U_n$. We define $\varphi\colon D_M^\alpha\to\omega$ by setting $\varphi(n^\smallfrown t)=\varphi_n(t)$ for each $t\in D_M^{\alpha_n}$. We claim that $\gseqn{V_{\varphi(s)}}{s\in D_M^\alpha}$ is a~$\fin^\alpha_M$-$\gamma$-cover of $G'$. Indeed, let $x\in G'$. By the~choice of~$\seqn{U_n}{n}$, there is an~$n_0$ such that $x\in U_n$ for each $n>n_0$. However, for each $n>n_0$ we have $\set{t\in D_M^{\alpha_n}}{x\not\in V_{\varphi_n(t)}}\in\fin^{\alpha_n}_M$. Hence, $\set{n^\smallfrown t\in D_M^\alpha}{x\not\in V_{\varphi(n^\smallfrown t)}}\in\fin^\alpha_M$.
\par
To prove the~reversed implication, we assume there is $\varphi\colon D_M^\alpha\to\omega$ such that $\gseqn{V_{\varphi(s)}}{s\in D_M^\alpha}$ is a~$\fin^\alpha_M$-$\gamma$-cover of~$G$. We set
\[
W_n=\set{x\in X}{\set{t\in D_M^{\alpha_n}}{x\not\in V_{\varphi(n^\smallfrown t)}}\in\fin^{\alpha_n}_M}.
\]
Thus $\gseqn{V_{\varphi(n^\smallfrown t)}}{t\in D_M^{\alpha_n}}$ is a~$\fin^{\alpha_n}_M$-$\gamma$-cover of~$W_n$, and we may use the~inductive hypothesis, so there is $W_n'\in\rmL^{\alpha_n}(\V)$ with $W_n'\supseteq W_n$. We set $G'=\Liminf\set{W_n'}{n\in\omega}$, and we shall show $G'\supseteq G$. If $x\in G$, then $\set{s\in D_M^\alpha}{x\not\in V_{\varphi(s)}}\in\fin^\alpha_M$. Thus there is an~$n_0$ such that $\set{t\in D_M^{\alpha_n}}{x\not\in V_{\varphi(n^\smallfrown t)}}\in\fin^{\alpha_n}_M$ for any $n>n_0$, so $x\in W_n\subseteq W_n'$ for any such~$n$. Then necessarily $x\in G'$.
\end{proof}
\par
Lemma~\ref{Lalpha-finalpha} provides a~criterion for a~set of reals $A\subseteq\PP(\omega)$ to judge it is not a~$\delta$-set.
\begin{lema}\label{nothavingdelta}
Let $A\subseteq\PP(\omega)$ possess \fup. If $A$ is a~$\delta$-set, then there is $\alpha<\omega_1$ such that $A\leqk\fin^\alpha$.   
\end{lema}
\begin{proof}
Let us assume that $A$ has \fup, and $\bigcup A=\omega$ (otherwise indeces run on~$\bigcup A$). Then $\V=\set{V_n}{n\in\omega}$ for $V_n=\set{a\in A}{n\not\in a}$ is an~$\omega$-cover of~$A$,
since no $V_n$ contains entire $A$. 
\par
Let $A$ have property~$\delta$. There is $\alpha<\omega_1$ such that $X\in\rmL^\alpha(\V)$. By Lemma~\ref{Lalpha-finalpha}, there is $\varphi\colon D^\alpha\to\omega$ such that $\gseqn{V_{\varphi(s)}}{s\in D^\alpha}$ is a~$\fin^\alpha$-$\gamma$-cover of~$A$. Thus for $a\in A$ we have $\varphi^{-1}(a)=\set{s\in D^\alpha}{a\not\in V_{\varphi(s)}}\in\fin^\alpha$. 
\end{proof}
\par
Another consequence of Lemma~\ref{Lalpha-finalpha} is an~alternative representation of the~hierarchy of $\delta$-sets. However, we need to recall the~definition of ideal $\gamma$-sets from~\cite{Su20}. A~topological space~$X$ is an~$\subselmgR{\J}$-space, or possesses the~property
\[
\subselmg{\J},
\] 
if for every open $\omega$-cover~$\V$ of~$X$ there is a~$\J$-$\gamma$-cover $\seqn{V_n}{n}$ of~$X$ with each~$V_n\in\V$. $\Omega$ in the~notation stands for the~family of all open $\omega$-covers of~$X$, and $\J\text{-}\Gamma$ is the~family of all $\J$-$\gamma$-covers of~$X$, see~\cite{Comb1,BDS}. Finally, let us state the~definition of a~$\gamma$-set and a~$\delta$-set via the~introduced hierarchy. Note that a~sequence $\seqn{V_n}{n}$ of subsets of~$X$ is a~$\fin$-$\gamma$-cover of~$X$ if and only if $\set{V_n}{n\in\omega}$ is a~$\gamma$-cover of~$X$. Thus, a~topological space~$X$ is a~$\gamma$-set if and only if $X\in\rmL^1(\V)$ for any open $\omega$-cover~$\V$. Moreover, $X$ has property~$\delta$ if and only if for any open $\omega$-cover~$\V$ there is $\alpha<\omega_1$ such that $X\in\rmL^\alpha(\V)$.
\begin{theorem}\label{subprinciple-deltalevels}
An~$\varepsilon$-space $X$ is an~$\subselmgR{\fin^\alpha}$-space if and only if $X\in\rmL^\alpha(\V)$ for any open $\omega$-cover~$\V$.
\end{theorem}
\begin{proof}
Taking $G=X$ in~Lemma~\ref{Lalpha-finalpha}, we obtain the~assertion.
\end{proof}
\par
\begin{corol}\label{subprinciple_katetov-delta}
If an~$\varepsilon$-space $X$ is an~$\subselmgR{\fin^\alpha}$-space then $X$ is a~$\delta$-set.
\end{corol}
\par
However, for our purposes, we need more versions of ideal $\gamma$-sets. Let $\square=\text{\rm 1-1}, \mathrm{KB}, \mathrm{K}$. A~topological space~$X$ is an~$\subselmgsqR{\J}$-space if for every open $\omega$-cover~$\V$ there is a~$\J$-$\gamma$-cover $\seqn{V_n}{n}$ such that each~$V_n$ belongs to~$\V$ and is repeated at most once, finitely many times, arbitrary many times in the~sequence $\seqn{V_n}{n}$, respectively. We immediately obtain
\[
\gamma\to\subselmgoo{\J}\to\subselmgkb{\J}\to\subselmgk{\J}.
\]
In case of $\subselmgkR{\J}$-space we usually abuse the~subscript and we write $\subselmgR{\J}$-space, as was introduced above. Note that an~$\varepsilon$-space~$X$ is an~$\subselmgsqR{\J}$-space if and only if for every countable open $\omega$-cover $\seqn{V_n}{n}$ there is a~$\square$-function $\varphi\in\baire\omega$ such that $\seqn{V_{\varphi(m)}}{m}$ is a~$\J$-$\gamma$-cover. In particular, we may assume here that~$\varphi$ is $\J$-to-one. For more, see~\cite{Su20}.
\par


\section{A~$\pi$-set that is not a~$\delta$-set}\label{S-pytkeev_nondelta}

M.~Sakai~\cite{SakQuad} has shown that any~$\delta$-set is a~$\pi$-set and raised the~following question. 
\begin{problem}[M.~Sakai]\label{prob-pi_delta}
Is every $\pi$-set also a~$\delta$-set?
\end{problem}
The~question was raised as Problem~4.2 in~\cite{SakQuad} (Question~3 in~\cite{Sak03} is also close). We answer the~question negatively, showing in~Theorem~\ref{pi-nondelta} that under~$\CH$ there is a~$\pi$-set of reals that is not a~$\delta$-set. However, we need more tools.
\par
For $\A\subseteq\PP(M)$ we denote $\A^\ast=\set{M\setminus A}{ A\in\A}$. A~family $\F\subseteq\PP(M)$ is a~filter if $\F^\ast$ is an~ideal. A~maximal filter $\U\subseteq\PP(M)$ is called an~ultrafilter. For an~ideal $\K\subseteq\PP(M)$ we set $\K^{+}=\PP(M)\setminus\K$. One can see that $B\in\K^+$ if and only if $M\setminus B\notin\K^\ast$.
\par
Let $b\in[\omega]^\omega$, then the~intervals $[b(j),b(j+1))$ form a~partition of~$\omega$. Such a~partition induces a~surjective finite to one map~$\bi_b\colon\omega\to\omega$ assigning to a~natural number the~index of the~interval it belongs to, i.e., $\bi_b(n)=j$ whenever $n\in[b(j),b(j+1))$.  Note that 
\[
\bi_b[a]=\set{j\in\omega}{[b(j),b(j+1))\cap a\neq\emptyset}.
\]
We say that 
\begin{itemize}
    \item $\I\leqrestk\J$ if there is $J\in\J^+$ such that $\I\leqk\J\restriction J$ (according to the~notation~$\leq_{\rm BE}^+$ in~\cite{farah}).
    \item $\I\leqcountk\J$ if there is a~function $f\in\baire\omega$ and a~countable family $B\subseteq\infin$ such that for any $a\subseteq\omega$ we have:
    \begin{center}
    if $\bi_b[a]\in\I$ for any $b\in B$, then $f^{-1}[a]\in\J$.
    \end{center}
\end{itemize}
Similarly to Kat\v etov order, the~function~$f$ in the~definition of~$\leqcountk$ has to be $\J$-to-one. We shall use this fact without any comment. One can check that if $\I_1\leqcountk\I_2$ and $\I_2\leqk\I_3$, then $\I_1\leqcountk\I_3$. Immediately from the~definition we have 
\[
\I\leqcountk\J\longleftarrow \I\leqk\J\longrightarrow \I\leqrestk\J.
\]
We include a~proof of auxiliary assertions. 
\begin{lema}\label{gfinnotlessi}
\begin{enumerate}[(1)]
    \item If $\I\not\leqrestk\J$, then for any $f\in\baire\omega$, $b\in\infin$, and $N\in\J^+$ there is $a\in[\omega]^\omega$ and $M\in\J^+\cap[N]^\omega$ such that $M\subseteq f^{-1}[a]$, and $\bi_b[a]\in\I$.
    \item If $\I\not\leqcountk\J$, then for any $\J$-to-one function $f\in\baire\omega$ and countable $B\subseteq\infin$ containing~$\omega$, there is an~infinite set $a\in\I$ such that $f^{-1}[a]\in\J^+$ and $\bi_b[a]\in\I$ for any $b\in B$.     
\end{enumerate}
\end{lema}
\begin{proof}
(1) We define $g\in\baire\omega$ by setting $g=\bi_b\circ f$, i.e., $g(m)=j$ whenever $f(m)\in[b(j),b(j+1))$. By $\I\not\leqk\J\restriction N$, there is an~infinite $c\in\I$ such that $M=(g\restriction N)^{-1}[c]\in\J^+\cap[N]^\omega$. We set $a=\bi_b^{-1}[c]=\bigcup\set{[b(j),b(j+1))}{j\in c}\in\infin$. One can see that $f^{-1}[a]=g^{-1}[c]\supseteq M$ and that $\bi_b[a]=c\in\I$.
\par
(2) Let $f\in\baire\omega$ be $\J$-to-one, $\omega\in B\subseteq\infin$, $|B|\leq\omega$. By $\I\not\leqcountk\J$, there is $a\in\infin$ such that $f^{-1}[a]\in\J^+$ and $\bi_b[a]\in\I$ for any $b\in B$. Moreover, $a=\bi_\omega[a]\in\I$.
\end{proof}
\par
A~$\pi$-set constructed to solve Problem~\ref{prob-pi_delta} will be an~$\subselmgR{\J}$-space for a~suitable ideal~$\J$. Therefore we need a~suitable criterion to judge that the~constructed set is an~$\subselmgR{\J}$-space. We say that the~family $\set{b_\xi}{\xi<\kappa}\subseteq\infin$ {\it codes representative covers} \label{evades} of~$\set{a_\xi}{\xi<\kappa}\subseteq\infin$ if each open $\omega$-cover of~$\set{a_\xi}{\xi<\kappa}\cup\fin$ contains a~subfamily~$\set{V_j}{j\in\omega}$ such that for some~$\eta<\kappa$ we have:
\begin{enumerate}[(1)]
    \item For each $j\in\omega$ we have $V_j\supseteq\set{a\subseteq\omega}{[b_{\eta}(j),b_{\eta}(j+1))\cap a=\emptyset}$,
    \item  $\set{V_j}{j\in\omega}$ is a~$\gamma$-cover of~$\set{a_\xi}{\xi<\eta}\cup\fin$.
\end{enumerate}
We assume that the~sequence $\seqn{V_j}{j}$ is injective. We say that the~family $\set{a_\xi}{\xi<\kappa}\subseteq\infin$ {\it $\J$-evades a~family} $\set{b_\xi}{\xi<\kappa}\subseteq\infin$ if for any $\alpha<\kappa$ we have 
    \begin{center}
    $\set{j\in\omega}{[b_\beta(j),b_\beta(j+1))\cap a_\alpha\neq\emptyset}\in\J$ for any $\beta\leq\alpha$.
    \end{center}
$\set{a_\xi,b_\xi}{\xi<\kappa}$ is {\it $\J$-evading double scale} if $\set{a_\xi}{\xi<\kappa}$ $\J$-evades family $\set{b_\xi}{\xi<\kappa}$ and $\set{b_\xi}{\xi<\kappa}$ codes representative covers of~$\set{a_\xi}{\xi<\kappa}$.
\begin{lema}\label{representativefamilies}
If $\set{a_\xi,b_\xi}{\xi<\kappa}$ is $\J$-evading double scale then the~set $A=\set{a_\alpha}{\alpha<\kappa}\cup\fin$ is an~$\subselmgooR{\J}$-space.
\end{lema}
\begin{proof}
To see that $A$ is an~$\subselmgooR{\J}$-space, let~$\V$ be an~open $\omega$-cover of~$A$. There is $\set{V_j}{j\in\omega}\subseteq\V$ that $\gamma$-covers $\set{a_\xi}{\xi<\eta}\cup\fin$ for some $\eta<\kappa$, and for each $j\in\omega$ we have 
\[
A\setminus V_j\subseteq\set{a\in A}{[b_\eta(j),b_\eta(j+1))\cap a\neq\emptyset}.
\]
By the~first property, $\seqn{V_j}{j}$ is a~$\J$-$\gamma$-cover of~$\set{a_\xi}{\xi<\eta}\cup\fin$. Furthermore, to see that $\seqn{V_j}{j}$ is a~$\J$-$\gamma$-cover of the~whole $A$, note that for $\alpha\geq\eta$ we have $\set{j\in\omega}{a_\alpha\not\in V_j}\subseteq\set{j\in\omega}{a_\alpha\cap[b_\eta(j),b_\eta(j+1))\neq\emptyset}\in\J$.
\end{proof}
\par
The~basic construction in our solution to Problem~\ref{prob-pi_delta} is accomplished in the~following assertion, where we construct a~$\J$-evading double scale of size~$\cc$.
\begin{theorem}[\CH]\label{subseqn-nond}
If $\J\not\leqcountk\fin^\alpha$ for each $\alpha<\omega_1$, then there is an~$\subselmgooR{\J}$-space $A\subseteq\PP(\omega)$ that is not a~$\delta$-set.
\end{theorem}
\begin{proof}
We fix a~bijection $\varrho\colon\cc\times\cc\to\cc$, and an~enumeration $\set{f_\alpha}{\alpha<\cc}$ of functions such that for each $\xi<\cc$, the~family $\set{f_{\varrho(\xi,\eta)}}{\eta<\cc}$ is the~family of all $\fin^\xi$-to-one functions from~${}^{D_\xi}\omega$. By transfinite induction, see the~last paragraph of the~proof, we shall construct  families $\set{a_\alpha}{\alpha<\cc}, \set{b_\alpha}{\alpha<\cc}\subseteq\infin$ such that $\set{b_\xi}{\xi<\cc}$ codes representative covers of~$\set{a_\xi}{\xi<\cc}$, and that the~conditions (i) - (iii) for $\xi,\eta<\cc$ with $\varrho(\xi,\eta)=\alpha$ are satisfied:
 \begin{enumerate}[(i)]
    \item $f_\alpha^{-1}[a_\alpha]\not\in\fin^\xi$,
    \item $a_\alpha\in\J$,
    \item $\set{j\in\omega}{[b_\beta(j),b_\beta(j+1))\cap a_\alpha\neq\emptyset}\in\J$ for any $\beta\leq\alpha$.
\end{enumerate}
Then $A=\set{a_\alpha}{\alpha<\cc}\cup\fin$ is an~$\subselmgR{\J}$-space. It follows by~Lemma~\ref{representativefamilies}, since the~family $\set{b_\xi}{\xi<\cc}$ codes representative covers of~$\set{a_\xi}{\xi<\cc}$ and condition~(iii) holds. Moreover, $A$ does not possess property~$\delta$. This is by Lemma~\ref{nothavingdelta}, since $A\not\leqk\fin^\xi$ for each $\xi<\cc$, and $A$ has \fup. 
\par
To proceed with the~construction, we fix an~enumeration $\set{\V_\alpha}{\alpha<\cc}$ of all countable open $\omega$-covers of~$\fin$ (elements of~$\V_\alpha$ are open sets in~$\PP(\omega)$). Once we have $\set{a_\beta}{\beta<\alpha}$, $\set{b_\beta}{\beta<\alpha}$ constructed, we consider~$\V_\alpha$. If~$\V_\alpha$ is not an~$\omega$-cover of $\set{a_\beta}{\beta<\alpha}\cup\fin$, then $b_\alpha$ may be arbitrary, e.g., $b_\alpha=\omega$. Otherwise, $\V_\alpha$ is an~$\omega$-cover of $\set{a_\beta}{\beta<\alpha}\cup\fin$, and we pick $b_\alpha\in\infin$ from Lemma~\ref{GMlemma}(2). Furthemore, by Lemma~\ref{gfinnotlessi}(2), we may pick~$a_\alpha\in\infin$ (considering $B=\set{b_\beta}{\beta\leq\alpha}\cup\{\omega\}$) such that the~conditions (i) - (iii) are satisfied, respectively.
\end{proof}
\par
By Theorem~\ref{subseqn-nond}, to solve Problem~\ref{prob-pi_delta}, it is enough to find an~appropriate ideal~$\J$ such that $\J\not\leqcountk\fin^\alpha$ for each $\alpha<\omega_1$, and that every $\subselmgR{\J}$-space is a~$\pi$-set. In fact, as we shall show later, the~appropriate ideal is any $\G^\ast\times\fin$ for arbitrary ultrafilter~$\G$. A~family~$A\subseteq\PP(\omega)$ is $\omega$-hitting if for any sequence~$\seqn{a_n}{n}$ of infinite subsets of~$\omega$ there is a~set~$a\in A$ such that $a_n\cap a$ is infinite for each $n\in\omega$. One can easily check that $\G^\ast\times\fin$ is not $\omega$-hitting, therefore the~following lemma is very useful.
\begin{lema}\label{subsequenceimpliespi}
Let $X$ be an~$\varepsilon$-space. If $\J$ is not $\omega$-hitting, then any~$\subselmgkbR{\J}$-space has property~$\pi$.
\end{lema}
\begin{proof}
Since $\J$ is not $\omega$-hitting there is a~sequence~$\seqn{J_n}{n}$ of infinite subsets of~$\omega$ such that for any set~$J\in\J$ there is $n_0\in\omega$ with $J_{n_0}\cap J\in\fin$.
\par
Let $\set{V_k}{k\in\omega}$ be an~$\omega$-cover. By the~assumption, there is a~finite-to-one function $\varphi\colon\omega\to\omega$ such that $\seqn{V_{\varphi(i)}}{i}$ is a~$\J$-$\gamma$-cover. We set~$\V_n=\set{V_{\varphi(i)}}{i\in J_n}$. Each~$\V_n$ is infinite. Moreover, for a~finite set~$F\subseteq X$ we have $\set{i}{F\not\subseteq V_{\varphi(i)}}\in\J$. Hence, there is $n_0$ such that $\set{i\in J_{n_0}}{F\not\subseteq V_{\varphi(i)}}$ is finite. 
\end{proof}
\begin{lema}\label{Gfin-notplus}
$\G^\ast\not\leqrestk\I$, $\G^\ast\times\fin\not\leqrestk\I$ for any ultrafilter~$\G$ and a~Borel ideal~$\I$.
\end{lema}
\begin{proof}
One can check that a~Borel ideal~$\I$ restricted to any $\I$-positive set is again Borel. Moreover, since $\G^\ast\leqk\G^\ast\times\fin$, we obtain $\G^\ast\times\fin\not\leqk\I$ as a~consequence of $\G^\ast\not\leqk\I$. Thus we shall prove just the~latter. To get the~contradiction, assume that there is a~function $\varphi\colon\omega\to\omega$ such that $\G^\ast\leq_\varphi\I$. One can check that the~induced map $F\colon\PP(\omega)\to\PP(\omega)$ defined by $F(Y)=\varphi^{-1}[Y]$ is continuous. By the~assumption $\G^\ast\leq_\varphi\I$, we immediately obtain $\G^\ast\subseteq F^{-1}[\I]$. Moreover, once we have $Y\subseteq\omega$ with $F(Y)\in\I$, we necessarily have $Y\in\G^\ast$ as well, since otherwise $F(Y)\in\I^\ast$ by the~maximality of~$\G$ and $\G^\ast\leq_\varphi\I$. Thus $\G^\ast=F^{-1}[\I]$, and therefore $\G^\ast$ is Borel, a~contradiction. 
\end{proof}
Let us recall that an~ideal $\J$ is a~$\rmP^+$-ideal if any decreasing sequence $\seqn{J_n}{n}$ of $\J$-positive sets has a~$\J$-positive pseudointersection. Moreover, according to~Section~\ref{S-katetov}, for a~sequence of ideals $\seqn{\J_n}{n}$ we denote $\J_\infty=\fin\text{-}\sum\set{\J_n}{n\in\omega}$.
\begin{prop}\label{onestepsubseqnonsubseq}
\begin{enumerate}[(1)]
    \item Let $\J$ be a~$\rmP^+$-ideal. If $\I\leqcountk\J$, then $\I\leqrestk\J$.
    \item Let $\I\not\leqcountk\J_n\restriction Y$ for any $Y\in\J_n^+$ and $n\in\omega$. If $\I\leqcountk\J_\infty$, then $\I\leqrestk\J_\infty$. 
\end{enumerate}
\end{prop}
\begin{proof}
(1) Suppose $\I\not\leqrestk\J$, we shall show that $\I\not\leqcountk\J$. Therefore, pick arbitrary $f\in\baire\omega$ and $B=\set{b_n}{n\in\omega}$. We need an~infinite set~$a$ with $f^{-1}[a]\in\J^+$ and $\bi_{b_n}[a]\in\I$ for each~$n$. To find such~$a$, we first construct inductively using Lemma~\ref{gfinnotlessi}(1), a~decreasing sequence $\set{A_n}{n\in\omega}\subseteq\J^+$ and a~sequence $\set{a_n}{n\in\omega}\subseteq\infin$ such that $A_n\subseteq f^{-1}[a_n]$, and 
\[
\bi_{b_n}[a_n]=\set{j\in\omega}{[b_n(j),b_n(j+1))\cap a_n\neq\emptyset}\in\I.
\]
Afterwards, we pick a~pseudointersection~$A\in\J^+$ of $\set{A_n}{n\in\omega}$. $A$ is a~pseudointersection of $\set{f^{-1}[a_n]}{n\in\omega}$ as well. Finally, we set $a=f[A]$. Then we have immediately $f^{-1}[a]\in\J^+$. Furthermore, $\bi_{b_n}[a]\in\I$ for each~$n$, since $a$ is almost contained in each $a_n$.
\par
(2) Suppose $\I\not\leqrestk\J_\infty$, we shall show that $\I\not\leqcountk\J_\infty$. Therefore, pick arbitrary function $f\colon\omega\times\omega\to\omega$ and the~set $B=\set{b_n}{n\in\omega}$. Our aim is to find an~infinite set~$a$ with $f^{-1}[a]\in\J_\infty^+$ and $\bi_{b_n}[a]\in\I$ for each $n\in\omega$. We set $d_{-1}=\omega$ and $m_{-1}=0$. 
\par
Taking $k\in\{-1\}\cup\omega$, we assume we have the~number~$m_k$, and the~infinite set $d_k\subseteq\omega$ with $u_k=f^{-1}[d_k]\in\J_\infty^+$. We shall consider the~map $f\restriction u_k:u_k\to d_k$ and the~set~$b_{k+1}$. By the~assumption $\I\not\leqrestk\J_\infty$ together with Lemma~\ref{gfinnotlessi}(1), we obtain an~infinite $d_{k+1}\subseteq d_k$ such that $(f\restriction u_k)^{-1}[d_{k+1}]\in\J_\infty^+$ and $\bi_{b_{k+1}}[d_{k+1}]\in\I$. Thus by the~definition of~$\J_\infty$, there is $m_{k+1}>m_k$ such that $c_{k+1}=\set{i}{(m_{k+1},i)\in (f\restriction u_k)^{-1}[d_{k+1}]}\in\J_{m_{k+1}}^+$. We set $h_{k+1}(i)=f(m_{k+1},i)$ for any $i\in c_{k+1}$. Then $h_{k+1}\colon c_{k+1}\to d_{k+1}$, and by the~assumption that $\I\not\leqcountk\J_{m_{k+1}}\restriction c_{k+1}$, we have an~infinite $a_{k+1}\subseteq d_{k+1}$ such that $h_{k+1}^{-1}[a_{k+1}]\in\J_{m_{k+1}}^+$, and 
\[
(\forall n\in\omega)\ \bi_{b_n}[a_{k+1}]\in\I.    
\]
\par
We set $a=\bigcup_{l\in\omega}a_l$. We claim that $f^{-1}[a]\in\J_\infty^+$ and $\bi_{b_n}[a]\in\I$ for each $n\in\omega$. Indeed, to prove the~first assertion, note that $\seqn{m_n}{n}$ is injective and $\set{j}{(m_n,j)\in f^{-1}[a]}\supseteq h_n^{-1}[a_n]\in\J_{m_n}^+$ for each $n\in\omega$. For the~second one, let us pick $n\in\omega$ and consider $\bi_{b_n}[a]$. We have 
\[
\bi_{b_n}[a]\subseteq \bi_{b_n}[d_n]\cup\bigcup_{l<n}\bi_{b_n}[a_l]\in\I.    
\]
\end{proof}
\par
\begin{lema}\label{Gfin-notcountable}
Let $0<\alpha<\omega_1$. Then $\G^\ast\times\fin\not\leqcountk\fin^\alpha$ for any ultrafilter~$\G$.
\end{lema}
\begin{proof}
We have $\G^\ast\times\fin\not\leqcountk\fin$ by Proposition~\ref{onestepsubseqnonsubseq}(1). Moreover, it follows that $\G^\ast\times\fin\not\leqcountk\fin\restriction Y$ for any infinite~$Y\subseteq\omega$. We continue by transfinite induction.
\par
Let $\alpha\geq2$, and $M\subseteq\baire\omega_1$ contain cofinal sequences. Since $\fin^\alpha$ is Kat\v etov equivalent to~$\fin^\alpha_M$, using Lemma~\ref{Gfin-notplus}, we conclude that $\G^\ast\times\fin\not\leqk\fin^\alpha_M$, and hence $\G^\ast\times\fin\not\leqrestk\fin^\alpha_M$ by Proposition~\ref{finbetaleqfinalpha}(4). By recursive assumption, we know that $\G^\ast\times\fin\not\leqcountk\fin^\beta_M$ for all $\beta<\alpha$, and again by Proposition~\ref{finbetaleqfinalpha}(4), we obtain $\G^\ast\times\fin\not\leqcountk\fin^\beta_M\restriction Y$ for $Y\not\in\fin^\beta_M$. Now a~direct application of Proposition~\ref{onestepsubseqnonsubseq}(2) gives $\G^\ast\times\fin\not\leqcountk\fin^\alpha_M$, and hence also $\G^\ast\times\fin\not\leqcountk\fin^\alpha$.
\end{proof}
\par
\begin{theorem}[\CH]\label{pi-nondelta}
There is a~$\pi$-set $A\subseteq\PP(\omega)$ that is not a~$\delta$-set. 
\end{theorem}
\begin{proof}
By Lemma~\ref{Gfin-notcountable}, we have $\G^\ast\times\fin\not\leqcountk\fin^\alpha$ for an~ultrafilter~$\G$. Thus by Theorem~\ref{subseqn-nond}, there is an~$\subselmgooR{(\G^\ast\times\fin)}$-space $A\subseteq\PP(\omega)$ that is not a~$\delta$-set. Moreover, $\G^\ast\times\fin$ is not $\omega$-hitting, therefore $A$ is a~$\pi$-set by~Lemma~\ref{subsequenceimpliespi}.
\end{proof}
\par


\section{Ideal $\gamma$-sets, and properties~$\gamma$, $\pi$, $\delta$}

\par
One can easily see that any $\gamma$-set is an~$\subselmgooR{\J}$-space. Naturally, if $\J$ is not \tall, then any $\subselmgooR{\J}$-space of reals is a~$\gamma$-set, see~Corollary~11.4 in~\cite{Su20}. However, we prove this is the~only case of such implication being true.
\begin{theorem}[$\pp=\cc$]\label{subselection-nongamma}
If $\J$ is \tall, then there is an~$\subselmgooR{\J}$-space $A\subseteq\PP(\omega)$ that is not a~$\gamma$-set. 
\end{theorem}
\begin{proof}
We fix an~enumeration $\set{c_\alpha}{\alpha<\cc}$ of all infinite subsets of~$\omega$, and a~\tall\ ideal~$\I$. By transfinite induction, we shall construct  families $\set{a_\alpha}{\alpha<\cc}, \set{b_\alpha}{\alpha<\cc}\subseteq\infin$ such that $\set{b_\xi}{\xi<\cc}$ codes representative covers of~$\set{a_\xi}{\xi<\cc}$, and conditions (1) - (3) are satisfied:
 \begin{enumerate}[(1)]
    \item $a_\alpha\subseteq c_\alpha$,
    \item $a_\alpha\in\I$,
    \item $\set{j\in\omega}{[b_\beta(j),b_\beta(j+1))\cap a_\alpha\neq\emptyset}\in\J$ for any $\beta\leq\alpha$.
\end{enumerate}
Then $A=\set{a_\alpha}{\alpha<\cc}\cup\fin$ is an~$\subselmgooR{\J}$-space but not a~$\gamma$-set. The~first fact follows by~Lemma~\ref{representativefamilies}, since the~family $\set{b_\xi}{\xi<\cc}$ codes representative covers of~$\set{a_\xi}{\xi<\cc}$ and condition~(3) holds. $A$ is \tall\ and has \fup, by conditions~(1), (2), respectively. Thus $A$ is not a~$\gamma$-set by Lemma~\ref{nothavinggamma}. 
\par
To proceed with construction, we fix an~enumeration $\set{\V_\alpha}{\alpha<\cc}$ of all countable open $\omega$-covers of~$\fin$ (elements of~$\V_\alpha$ are open sets in~$\PP(\omega)$). Once we have $\set{a_\beta}{\beta<\alpha}$, $\set{b_\beta}{\beta<\alpha}$ constructed, we consider~$\V_\alpha$. If~$\V_\alpha$ is not an~$\omega$-cover of $\set{a_\beta}{\beta<\alpha}\cup\fin$, then $b_\alpha$ can be arbitrary, e.g., we set $b_\alpha=\omega$. Otherwise, $\V_\alpha$ is an~$\omega$-cover of $\set{a_\beta}{\beta<\alpha}\cup\fin$, and we pick $b_\alpha\in\infin$ such as in Lemma~\ref{GMlemma}(2). \par
To pick~$a_\alpha\in\infin$ such that conditions (1) - (3) are satisfied, it is enough to construct $\subseteq^\ast$-decreasing sequence $\gseqn{a_{\alpha,\beta}}{\beta\leq\alpha}$ such that $c_\alpha\supseteq a_{\alpha,\beta}$, $a_{\alpha,\beta}\in\I$, and 
\[
\set{j\in\omega}{[b_\beta(j),b_\beta(j+1))\cap a_{\alpha,\beta}\neq\emptyset}\in\J.
\]
Indeed, then we just pick a~pseudointersection $a_\alpha\subseteq c_\alpha$ of the~sequence $\gseqn{a_{\alpha,\beta}}{\beta\leq\alpha}$. To construct the~sequence $\gseqn{a_{\alpha,\beta}}{\beta\leq\alpha}$, let us consider $\beta\leq\alpha$. For $\beta=\xi+1$ it is enough to apply the~claim below to~$b_\beta$ and~$a_\xi$. If $\beta$ is limit, then we select a~pseudointersection~$d\subseteq c_\alpha$ of $\gseqn{a_{\alpha,\xi}}{\xi\leq\beta}$, and afterwards we apply the~claim below to~$b_\beta$ and~$d$.
\begin{clm*}
Let~$b,c\in\infin$. There is $a\in[c]^\omega\cap\I$ such that  
\[
\set{j\in\omega}{[b(j),b(j+1))\cap a\neq\emptyset}\in\J.    
\]
\end{clm*}
\begin{proof}
Let $N=\set{j\in\omega}{[b(j),b(j+1))\cap c\neq\emptyset}\in\infin$. There is $M\in[N]^\omega$ with $M\in\J$. Hence, we pick an~infinite subset~$a$ of $c\cap\bigcup_{n\in M}[b(j),b(j+1))$ such that $a\in\I$. 
\end{proof}
\end{proof}
\par
Corollary~\ref{subprinciple_katetov-delta} states that in the~realm of $\varepsilon$-spaces, any~$\subselmgR{\fin^\alpha}$-space possesses the~property~$\delta$, and so  the~property~$\pi$ as well. Moreover, we have shown in~Lemma~\ref{subsequenceimpliespi} that in the~realm of $\varepsilon$-spaces, if $\J$ is not $\omega$-hitting, then any~$\subselmgkbR{\J}$-space has the~property~$\pi$. Let us recall that the~nowhere dense ideal~$\nwd$ is the~ideal on the~set of rational numbers~$\Q$ whose elements are the~nowhere dense subsets of~$\Q$.  
\begin{prop}\label{subsequenceimpliespidelta}
Let $X$ be an~$\varepsilon$-space.  
\begin{enumerate}[(1)]
    \item If $\J\leqk\fin^\alpha$, then any~$\subselmgR{\J}$-space has property $\delta$. 
    \item If $\J\leqk\nwd$, then any~$\subselmgR{\J}$-space has property~$\pi$.
\end{enumerate}
\end{prop}
\begin{proof}
(1) A~direct consequence of Theorem~\ref{subprinciple-deltalevels}, since any~$\subselmgR{\J}$-space with $\J\leqk\fin^\alpha$ is an~$\subselmgR{\fin^\alpha}$-space.
\par
(2) Similarly to part~(1), it is enough to show that the~statement holds for $\J=\nwd$. Let $\set{V_k}{k\in\omega}$ be an~$\omega$-cover. By the~assumption, there is a~function $\varphi\colon\Q\to\omega$ such that $\langle V_{\varphi(q)}\colon q\in\Q\rangle\in\iGamma{\nwd}$. We may assume that $\varphi$ is $\nwd$-to-one, then $\set{\varphi(q)}{q\in B}$ is infinite for any open~$B\subseteq\Q$. We fix a~countable basis~$\B$ of topology~$\Q$. We shall show that each finite set $F\subseteq X$ is contained in $V_{\varphi(q)}$ for all $q\in B$, for some $B\in\B$. Indeed, let $F\subseteq X$ be finite. Since $\set{q\in\Q}{F\not\subseteq V_{\varphi(q)}}\in\nwd$ there is $B\in\B$ such that $\set{q\in\Q}{F\not\subseteq V_{\varphi(q)}}\cap B=\emptyset$. Hence, $F\subseteq V_{\varphi(q)}$ for all $q\in B$.
\end{proof}
\par
It could be tempting to conjecture that any~$\subselmgkbR{\J}$-space has the~property~$\pi$. However, as we shall see in Theorem~\ref{subsequence-nonpi}, it is not true. Let us recall that some standard Borel ideals on~$\omega$ studied in~\cite{BrFl17,BrFaVe,Hr,Hr17}, namely $\EDfin$, $\summable$, $\Z$, are all $\omega$-hitting. 
\begin{lema}\label{nothavingproperties}
Let $A\subseteq\PP(\omega)$ possess \fup. If $A$ is $\omega$-hitting, then $A$ is not a~$\pi$-set.
\end{lema}
\begin{proof}
Let us assume that $A$ has \fup, and $\bigcup A=\omega$ (otherwise indeces run on~$\bigcup A$). Then $\V=\set{V_n}{n\in\omega}$ for $V_n=\set{a\in A}{n\not\in a}$ is an~$\omega$-cover of~$A$,
since no $V_n$ contains entire $A$.
\par
To get the~contradiction, let us assume that $A$ is a~$\pi$-set. Thus there is a~sequence $\seqn{\V_k}{k}$ of infinite subsets of~$\V$ such that for any finite set~$F\subseteq A$ there is $m\in\omega$ with $F\subseteq V$ for all but finitely many $V\in\V_m$. Taking $c_k=\set{n\in\omega}{V_n\in\V_k}$ we obtain a~sequence~$\seqn{c_k}{k}$ of infinite subsets of~$\omega$ such that for any finite set~$F\subseteq A$ there is $m\in\omega$ with $(\bigcup F)\cap c_m$ finite. Let $a\in A$ be such that $|a\cap c_k|=\omega$ for each~$k\in\omega$. However, for $F=\{a\}$ we have $|a\cap c_m|<\omega$ for some~$m$, a~contradiction.
\end{proof}
\par
\begin{theorem}[\CH]\label{subsequence-nonpi}
The~following are equivalent for any ideal~$\I$ on~$\omega$:
\begin{enumerate}[\rm (a)]
    \item $\I$ is not $\omega$-hitting.
    \item Any~$\subselmgooR{\I}$-space with property~$\varepsilon$ has property~$\pi$.
\end{enumerate}
\end{theorem}
\begin{proof}
$\text{\rm (a)}\to\text{\rm (b)}$ follows from Lemma~\ref{subsequenceimpliespi}.
\par
$\text{\rm (b)}\to\text{\rm (a)}$ We shall show that if $\I$ is $\omega$-hitting, then there is an~$\subselmgooR{\I}$-space $A\subseteq\PP(\omega)$ that does not have property~$\pi$. Fix an~enumeration $\set{\set{c^\alpha_n}{n\in\omega}}{\alpha<\cc}$ of all countable families of infinite subsets of~$\omega$, and an~$\omega$-hitting ideal~$\J$. By transfinite induction, we shall construct  families $\set{a_\alpha}{\alpha<\cc}, \set{b_\alpha}{\alpha<\cc}\subseteq\infin$ such that $\set{b_\xi}{\xi<\cc}$ codes representative covers of~$\set{a_\xi}{\xi<\cc}$, and that the~conditions (1) - (3) are satisfied:
\begin{enumerate}[(1)]
    \item $|a_\alpha\cap c^\alpha_m|=\omega$ for each~$m\in\omega$,
    \item $a_\alpha\in\J$,
    \item $\set{n\in\omega}{[b_\beta(n),b_\beta(n+1))\cap a_\alpha\neq\emptyset}\in\I$ for each~$\beta\leq\alpha$.
\end{enumerate}
Then $A=\set{a_\alpha}{\alpha<\cc}\cup\fin$ is an~$\subselmgooR{\I}$-space but does not have property~$\pi$. The~first fact follows by~Lemma~\ref{representativefamilies}, since the~family $\set{b_\xi}{\xi<\cc}$ codes representative covers of~$\set{a_\xi}{\xi<\cc}$, and condition~(3) holds. $A$ is $\omega$-hitting and has \fup, by conditions~(1) and~(2), respectively. Thus $A$  does not have property~$\pi$ by Lemma~\ref{nothavingproperties}. 
\par
To proceed with construction, we fix an~enumeration $\set{\V_\alpha}{\alpha<\cc}$ of all countable open $\omega$-covers of~$\fin$ (elements of~$\V_\alpha$ are open sets in~$\PP(\omega)$). Once we have $\set{a_\beta}{\beta<\alpha}$, $\set{b_\beta}{\beta<\alpha}$ constructed, we consider~$\V_\alpha$. If~$\V_\alpha$ is not an~$\omega$-cover of $\set{a_\beta}{\beta<\alpha}\cup\fin$, then $b_\alpha$ and $\set{V_{\alpha,j}}{j\in\omega}$ may be arbitrary, e.g., identity and $\V_\alpha$ itself, respectively. Otherwise, $\V_\alpha$ is an~$\omega$-cover of $\set{a_\beta}{\beta<\alpha}\cup\fin$, and we pick $b_\alpha\in\infin$ from Lemma~\ref{GMlemma}(2). Furthemore, by the~following claim, we may pick~$a_\alpha\in\infin$ such that the~conditions (1) - (3) are satisfied.
\begin{clm*}
Let $B, C\subseteq\infin$, $|B|=|C|=\omega$. Then there is $a\in\J\cap\infin$ such that $|a\cap c|=\omega$ for each~$c\in C$, and 
\begin{center}
$(\forall b\in B)\ \set{n\in\omega}{[b(n),b(n+1))\cap a\neq\emptyset}\in\I$.    
\end{center}
\end{clm*}
\begin{proof}
We fix an~enumeration $\set{b_i}{i\in\omega}=B$ and $\set{c_i}{i\in\omega}=C$. Since~$\J$ is $\omega$-hitting there is an~infinite $a'\in\J$ such that $c_m\cap a'$ is infinite for each~$m$. We set 
\[
u^0_m=\set{n\in\omega}{[b_0(n),b_0(n+1))\cap c_m\cap a'\neq\emptyset}\in\infin.
\] 
Since~$\I$ is $\omega$-hitting there is an~infinite $d_0\in\I$ such that $u^0_m\cap d_0$ is infinite for each~$m$. We may assume that sets $u^0_m\cap d_0$ are mutually disjoint for $m\in\omega$. We pick $k(0,m,n)\in[b_0(n),b_0(n+1))\cap c_m\cap a'$ for each $m\in\omega$ and $n\in u^0_m\cap d_0$, and we set $a_0=\set{k(0,m,n)}{n\in u^0_m\cap d_0, m\in\omega}\in\J$, the~latter inclusion follows from $a_0\subseteq a'$.
\par
We assume that $a_0\supseteq \dots\supseteq a_j$ are constructed and possess the following properties: 
\begin{center}
    $\set{n\in\omega}{[b_i(n),b_i(n+1))\cap a_i\neq\emptyset}\in\I$ and $|a_i\cap c_m|=\omega$ for each~$i\leq j$, $m\in\omega$.
\end{center}
We set $u^{j+1}_m=\set{n\in\omega}{[b_{j+1}(n),b_{j+1}(n+1))\cap c_m\cap a_j\neq\emptyset}\in\infin$. We continue as in the~previous case to obtain an~infinite $d_{j+1}\in\I$ such that $u^{j+1}_m\cap d_{j+1}$ are infinite, mutually disjoint for $m\in\omega$. We pick $k(j+1,m,n)\in[b_{j+1}(n),b_{j+1}(n+1))\cap c_m\cap a_j$ for each $m\in\omega$ and $n\in u^{j+1}_m\cap d_{j+1}$, and we set $a_{j+1}=\set{k(j+1,m,n)}{n\in u^{j+1}_m\cap d_{j+1}, m\in\omega}\in\J$. Finally, having $\seqn{a_i}{i}$ defined, we take a~pseudointersection~$a$ of $\seqn{a_i}{i}$ such that $c_m\cap a$ is infinite for each~$m$. One can check that the~requirements on~$a$ are fulfilled.  
\end{proof}
\end{proof}
\par


\section{Ideal $\gamma$-sets for different ideals}\label{S-gamma_sets}

\par
Although we have already studied ideal $\gamma$-sets in previous sections, the~ideals under considerations were mainly Kat\v etov power ideals and ideal $\G^\ast\times\fin$ (for an~ultrafilter~$\G$). The~aim of the present section is to consider standard Borel ideals on~$\omega$ studied in~\cite{BrFl17,BrFaVe,Hr,Hr17}. The~corresponding diagram of their relations together with the~Kat\v etov power ideals is depicted in Diagram~\ref{diagram:katetovideals}. Here an~arrow means $\leqk$.
\par
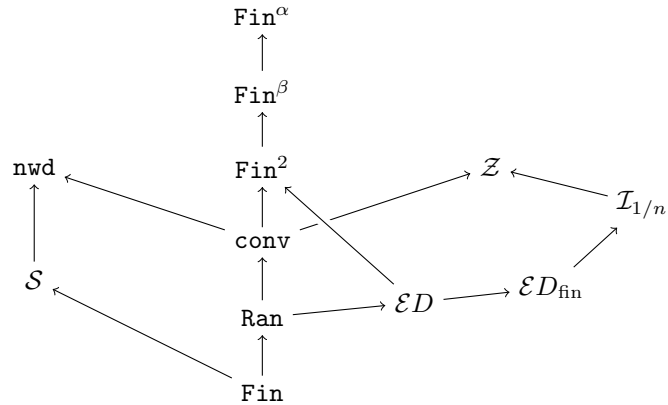
\begin{figure}[H]
\begin{center}
\begin{tikzpicture}[]
\node (g) at (-0.5, -3.5) {$\fin$};
\node (r) at (-0.5, -2.5) {$\Ran$};
\node (conv) at (-0.5, -1.5) {$\conv$};
\node (f2) at (-0.5, -0.5) {$\fin^2$};
\node (fbeta) at (-0.5, 0.5) {$\fin^\beta$};
\node (falpha) at (-0.5, 1.5) {$\fin^\alpha$};
\node (s) at (-3.5, -2) {$\calS$};
\node (nwd) at (-3.5, -0.5) {$\nwd$};
\node (ed) at (1.5, -2.3) {$\ED$};
\node (edfin) at (3.3, -2.1) {$\EDfin$};
\node (sum) at (4.5, -1) {$\summable$};
\node (z) at (2.5, -0.5) {$\Z$};
\foreach \from/\to in {conv/z} \draw [->] (\from) -- (\to);
\foreach \from/\to in {ed/f2} \draw [line width=.15cm,white] (\from) -- (\to);
\foreach \from/\to in {g/r,g/s,r/conv,conv/f2,f2/fbeta,fbeta/falpha,conv/nwd,s/nwd,r/ed,ed/f2,ed/edfin,edfin/sum,sum/z} \draw [->] (\from) -- (\to);
\end{tikzpicture}
\end{center}
\caption{Borel ideals studied in~\cite{BrFl17,BrFaVe,Hr,Hr17}. Here $2\leq\beta<\alpha<\omega_1$.}
\label{diagram:katetovideals}
\end{figure}
\par
By Proposition~11.3 in~\cite{Su20}, if $\I_1\leqk\I_2$, then any~$\subselmgR{\I_1}$-space is an~$\subselmgR{\I_2}$-space as well. Thus, Diagram~\ref{diagram:katetovideals} gives rise to relations among corresponding ideal $\gamma$-sets in Diagram~\ref{diagram:katetov}. It is natural to ask about additional implications between properties from~Diagram~\ref{diagram:katetov}. We dedicate the~rest of the~section to such investigations. In fact, we show that most of the~missing implications cannot be added. The~proof is based on the~following theorem.  
\begin{theorem}[\CH]\label{subseqn-nonsubseqn}
If $\set{\J_\alpha}{\alpha<\omega_1}$ is a~family of ideals such that $\I\not\leqcountk\J_\alpha$ for all $\alpha<\omega_1$, then there exists an~$\subselmgooR{\I}$-space $A\subseteq\PP(\omega)$ that is not an~$\subselmgR{\J_\alpha}$-space for all $\alpha<\omega_1$. 
\end{theorem}
\begin{proof}
We fix a~bijection $\varrho\colon\cc\times\cc\to\cc$, and an~enumeration $\set{f_\alpha}{\alpha<\cc}$ of functions such that for each $\xi<\cc$, the~family $\set{f_{\varrho(\xi,\eta)}}{\eta<\cc}$ is the~family of all $\J_\xi$-to-one functions from~$\baire\omega$. By transfinite induction, see the~last paragraph of the~proof, we shall construct  families $\set{a_\alpha}{\alpha<\cc}, \set{b_\alpha}{\alpha<\cc}\subseteq\infin$ such that $\set{b_\xi}{\xi<\cc}$ codes representative covers of~$\set{a_\xi}{\xi<\cc}$, and that the~conditions (i) - (iii) for $\xi,\eta<\cc$ with $\varrho(\xi,\eta)=\alpha$ are satisfied:
 \begin{enumerate}[(i)]
    \item $f_\alpha^{-1}[a_\alpha]\in\J_\xi^+$,
    \item $a_\alpha\in\I$,
    \item $\set{j\in\omega}{[b_\beta(j),b_\beta(j+1))\cap a_\alpha\neq\emptyset}\in\I$ for any $\beta\leq\alpha$.
\end{enumerate}
Then $A=\set{a_\alpha}{\alpha<\cc}\cup\fin$ is an~$\subselmgooR{\I}$-space. It follows by~Lemma~\ref{representativefamilies}, since the~family $\set{b_\xi}{\xi<\cc}$ codes representative covers of~$\set{a_\xi}{\xi<\cc}$ and condition~(iii) holds.
\par
$A$ is not an~$\subselmgR{\J_\xi}$-space for all $\xi<\cc$. Indeed, let $\xi<\cc$. For any $\J_\xi$-to-one function $f\in\baire\omega$ there is $\eta<\cc$ such that $f=f_{\varrho(\xi,\eta)}$. By condition~(i), we have $f^{-1}[a_{\varrho(\xi,\eta)}]\in\J_\xi^+$. Thus $A\not\leqk\J_\xi$. Moreover, $A$ has \fup\ by condition~(ii). By Lemma~3.2(2) in~\cite{Su20}, no $A\not\leqk\J_\xi$ with \fup\ is an~$\subselmgR{\J_\xi}$-space.
\par
To proceed with construction, we fix an~enumeration $\set{\V_\alpha}{\alpha<\cc}$ of all countable open $\omega$-covers of~$\fin$ (elements of~$\V_\alpha$ are open sets in~$\PP(\omega)$). Once we have $\set{a_\beta}{\beta<\alpha}$, $\set{b_\beta}{\beta<\alpha}$ constructed, we consider~$\V_\alpha$. If~$\V_\alpha$ is not an~$\omega$-cover of $\set{a_\beta}{\beta<\alpha}\cup\fin$, then $b_\alpha$ may be arbitrary, e.g., $b_\alpha=\omega$. Otherwise, $\V_\alpha$ is an~$\omega$-cover of $\set{a_\beta}{\beta<\alpha}\cup\fin$, and we pick $b_\alpha\in\infin$ from Lemma~\ref{GMlemma}(2). Furthemore, by Lemma~\ref{gfinnotlessi}(2), we may pick~$a_\alpha\in\infin$ (considering $B=\set{b_\beta}{\beta\leq\alpha}\cup\{\omega\}$) such that the~conditions (i) - (iii) are satisfied, respectively.
\end{proof}
\par
Consequently, by Proposition~\ref{onestepsubseqnonsubseq} we have 
\begin{corol}[\CH]\label{subsequence-nonsubsequence_Pplus}
If $\J$ is a~$\rmP^+$-ideal and $\I\not\leqrestk\J$, then there is an~$\subselmgooR{\I}$-space $A\subseteq\PP(\omega)$ that is not an~$\subselmgR{\J}$-space. 
\end{corol}
\par
In the~wave of Theorem~\ref{subseqn-nonsubseqn}, we get interested in ideals~$\I,\J$ such that $\I\not\leqcountk\J$. Focusing on Diagram~\ref{diagram:katetovideals}, it is known that with one possible exception (it is not known whether $\Ran\leqk\calS$, see~\cite{BrFl17,BrFaVe,Hr,Hr17}), each missing arrow from~$\I$ to~$\J$ means that $\I\not\leqk\J$ (under the~assumption that we added all arrows arising from transitivity). Although $\finfin\leqrestk\conv$ and $\finfin\not\leqk\conv$, in many other cases in Diagram~\ref{diagram:katetovideals} we have that if $\I\not\leqk\J$, then $\I\not\leqrestk\J$. Thus the~answer to the~following question about ideals~$\I,\J$ becomes important:
\begin{center}
If $\I\leqcountk\J$, is it necessary true that $\I\leqrestk\J$ (or even $\I\leqk\J$) as well?    
\end{center}
\par
In Proposition~\ref{onestepsubseqnonsubseq}, we have already answered the~question positively for some classes of ideals. We add an~additional class of ideals. Let us recall that the~ideal~$\conv$ is the~ideal on the~set of rational numbers~$\Q\cap[0,1]$ induced by

sequences convergent in~$[0,1]$.
\begin{lema}\label{omegaisplus_nwd}
If $\I\leqcountk\conv$, then $\I\leqrestk\conv$. 
\end{lema}
\begin{proof}
Suppose $\I\not\leqrestk\conv$ (in particular, $\I$ is tall), we shall show that $\I\not\leqcountk\conv$. Therefore, pick arbitrary $f\in\baire\omega$ and $B=\set{b_n}{n\in\omega}$. By Lemma~\ref{gfinnotlessi}(1), there is $a_0\in[\omega]^\omega$ such that $f^{-1}[a_0]\not\in\conv$, and 
\[
\bi_{b_0}[a_0]=\set{j\in\omega}{[b_0(j),b_0(j+1))\cap a_0\neq\emptyset}\in\I.
\]
$f^{-1}[a_0]\not\in\conv$, so it has infinitely many accumulation points. Let us pick a~bijective enumeration $\seqn{r_i}{i}$ of all accumulation points of~$f^{-1}[a_0]$, and $A_{0,0}\in[f^{-1}[a_0]]^\omega$ that converges to~$r_0$. Note that $f[A_{0,0}]\subseteq a_0$, and so $\bi_{b_0}\big[f[A_{0,0}]\big]\in\I$. 
\par
By Lemma~\ref{gfinnotlessi}(1), there is $a_1\in[a_0]^\omega$ such that $f^{-1}[a_1]\not\in\conv$ and $\bi_{b_1}[a_1]\in\I$. Let us pick $A_{1,1}\in[f^{-1}[a_1]]^\omega$ that converges to~$r_{i_1}$ with $i_1=\min\set{i}{r_i\in\overline{f^{-1}[a_1]}, i>0}$. Note that $f[A_{1,1}]\subseteq a_1$, and so $\bi_{b_1}\big[f[A_{1,1}]\big]\in\I$. 
\par
We assume that we have a~decreasing sequence $a_0,\dots, a_k$ of infinite sets such that $f^{-1}[a_i]\not\in\conv$ and $\bi_{b_i}[a_i]\in\I$. Moreover, we assume we have $\gseqn{A_{n,n}}{n\leq k}$, and increasing $\gseqn{i_n}{n\leq k}$ such that $A_{n,n}\in[f^{-1}[a_n]]^\omega$ is converging to~$r_{i_n}$ and $\bi_{b_n}\big[f[A_{n,n}]\big]\in\I$.
\par
By Lemma~\ref{gfinnotlessi}(1), there is $a_{k+1}\in[a_k]^\omega$ such that $f^{-1}[a_{k+1}]\not\in\conv$ and $\bi_{b_{k+1}}[a_{k+1}]\in\I$. Let us pick $A_{k+1,k+1}\in[f^{-1}[a_{k+1}]]^\omega$ that converges to~$r_{i_{k+1}}$ with $i_{k+1}=\min\set{i}{r_i\in\overline{f^{-1}[a_{k+1}]}, i>i_k}$. Note that $f[A_{k+1,k+1}]\subseteq a_{k+1}$, and so $\bi_{b_{k+1}}\big[f[A_{{k+1},k+1}]\big]\in\I$. 
\par
Taking into account that the ideal $\I$ is $\tall$ (note that $\J\leqrestk\conv$ for $\J$ being not tall), one can construct a~decreasing family $\set{A_{n,m}}{m\in\omega}$ of infinite sets with the~largest element $A_{n,n}$ such that $\bi_{b_m}\big[f[A_{n,m}]\big]\in\I$. We take a~pseudointersection $A_n$ of $\seqn{A_{n,m}}{m}$ below $A_{n,n}$, and set $A=\bigcup_{n\in\omega}A_n$. One can see that $A\not\in\conv$. Moreover, $f[A]\subseteq a_0$. Finally,
\[
\bi_{b_k}\big[f[A]\big]\subseteq\bi_{b_k}[a_k]\cup\bigcup\limits_{i<k}\bi_{b_k}\big[f[A_i]\big]\in\I.
\]
Thus $f$ and $B$ do not witness $\I\leqcountk\conv$, as exemplified by~$f[A]$. Since $f$ and $B$ were chosen arbitrarily, we have $\I\not\leqcountk\conv$.
\end{proof}
\par
It is well-known that ideals $\fin$, $\summable$, $\ED$, $\EDfin$, $\Ran$, $\calS$ are all $\Fsigma$ ideals. Moreover, by the~result attributed to~\cite{JustKra} in~\cite{Hr}, any $\Fsigma$ ideal is a~$\rmP^+$-ideal, i.e., any decreasing sequence $\seqn{I_n}{n}$ of $\I$-positive sets has an~$\I$-positive pseudointersection. 
\begin{lema}\label{katetov-rest}
\begin{enumerate}[(1)]
    \item Let $0<\alpha<\omega_1$. If $\I\leqcountk\fin^\alpha$, then $\I\leqk\fin^\alpha$. In particular, $\fin^\alpha\not\leqcountk\fin^\beta$ for any $0<\beta<\alpha$.
    \item $\finfin\not\leqcountk\J$ for any $\rmP^+$-ideal~$\J$. In particular, $\finfin\not\leqcountk\fin, \summable, \ED, \EDfin, \Ran, \calS$.
    \item $\nwd\not\leqcountk\J$ for any $\Fsigma$ ideal~$\J$. In particular, $\nwd\not\leqcountk\fin, \summable, \ED, \EDfin, \Ran, \calS$.
    \item $\I\not\leqcountk\conv$ for any $\I\not\leqk\finfin$. In particular, $\nwd, \calS, \Z, \summable, \EDfin\not\leqcountk\conv$.
    \item $\I\not\leqcountk\J$  for any $\Fsigma$ ideal~$\J$ such that $\conv\leqk\I$. In particular, 
    \begin{center}
        $\conv, \fin^\alpha\not\leqcountk\calS, \Ran, \ED, \EDfin, \summable$ ($\alpha>1$).
    \end{center}
    \item $\fin^\alpha_M\not\leqcountk\fin^\beta_M$ for $0<\beta<\alpha<\omega_1$. $\summable\not\leqcountk\EDfin$.        
\end{enumerate}
\end{lema}
\begin{proof}
In cases (2) - (6), we prove the~statement just for the~relation~$\leqrestk$. The~corresponding statement for the~relation~$\leqcountk$ then follows by Proposition~\ref{onestepsubseqnonsubseq} and Lemma~\ref{omegaisplus_nwd}.
\par
(1) If $\I\not\leqk\fin$, then $\I\not\leqcountk\fin$ by Proposition~\ref{onestepsubseqnonsubseq}(1). Moreover, it follows that $\I\not\leqcountk\fin\restriction Y$ for any infinite~$Y\subseteq\omega$. We continue by transfinite induction.
\par
Suppose that $\I\not\leqk\fin^\alpha$ for some $\alpha\geq2$. Let $M\subseteq\baire\omega_1$ contains cofinal sequences. Since $\fin^\alpha$ is Kat\v etov equivalent to~$\fin^\alpha_M$, we conclude that $\I\not\leqk\fin^\alpha_M$, and hence $\I\not\leqrestk\fin^\alpha_M$ by Proposition~\ref{finbetaleqfinalpha}(4). By recursive assumption, we know that $\I\not\leqcountk\fin^\beta_M$ for all $\beta<\alpha$ (beacause $\I\not\leqk\fin^\beta_M$), and again by Proposition~\ref{finbetaleqfinalpha}(4), we obtain $\I\not\leqcountk\fin^\beta_M\restriction Y$ for $Y\not\in\fin^\beta_M$. Now a~direct application of Proposition~\ref{onestepsubseqnonsubseq}(2) gives $\I\not\leqcountk\fin^\alpha_M$, and hence also $\I\not\leqcountk\fin^\alpha$.
\par
The~particular case follows by $\fin^\alpha\not\leqk\fin^\beta$, see Proposition~\ref{finbetaleqfinalpha}. 
\par
(2) Proposition~5.11 in~\cite{Hr}.
\par
(3) Note that $\J\restriction J$ is an~$\Fsigma$ ideal for each $J\in\J^+$. However, $\nwd\not\leqk\K$ for any $\Fsigma$ ideal~$\K$. Indeed, by~\cite[Corollary~2.4]{HrMeMi10}, it is consistent that $\covh{\K}>\dd$. On the~other hand, $\covh{\nwd}=\covm\leq\dd$. Thus if $\nwd\leqk\K$ was true, then it would be true in the~above mentioned model by the~absolutness of Kat\v etov order, and we obtain $\covh{\nwd}\geq\covh{\K}$, a~contradiction.
\par
(4) Each $J\not\in\conv$ contains $\conv$-positive subset $J'\subseteq J$ such that $\conv\restriction J'$ is isomorphic to $\finfin$.
\par
(5) No $\Fsigma$ ideal is Kat\v etov above $\conv$, see~\cite{Hr17}. Considering $J\in\J^+$, one can check that $\J\restriction J$ is an~$\Fsigma$ ideal. Since $\conv\leqk\I$, we have $\I\not\leqk\J\restriction J$.
\par
(6) By Proposition~\ref{finbetaleqfinalpha}, $\fin^\beta_M$ is K-equivalent to $\fin^\beta_M\restriction Y$ for each $Y\not\in\fin^\beta_M$, and $\fin^\alpha_M\not\leqk\fin^\beta_M$.
\par
To show that $\summable\not\leqrestk\EDfin$, note that~$\summable\not\leqk\EDfin$, see~\cite{Hr17}. If $E\in\EDfin^+$, then there is an~infinite set~$I\subseteq\omega$ such that the~number of elements of $(I(j)\times\omega)\cap E$ is increasing with~$j$. Setting $E'=(I\times\omega)\cap\Delta\in\EDfin^+$ for $\Delta=\set{(n,m)}{n\leq m}$, we obtain $\EDfin\restriction E\leqk\EDfin\restriction E'\leqk\EDfin$.
\end{proof}
\par
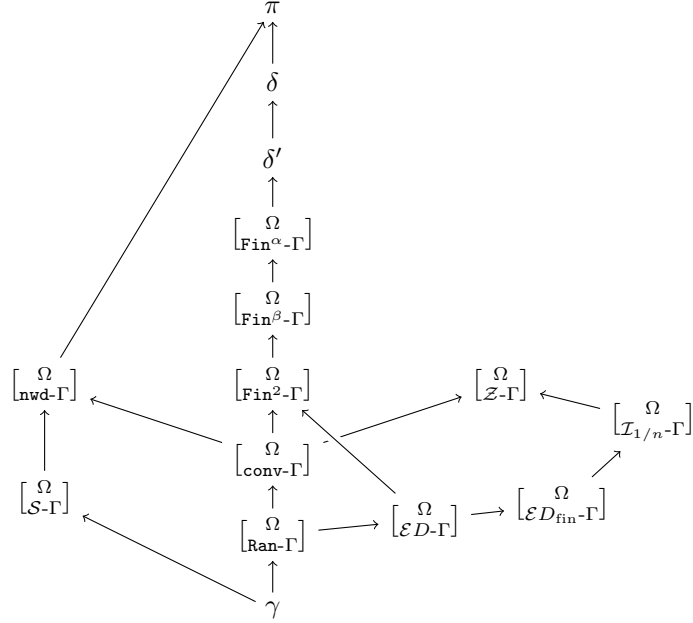
\begin{figure}[H]
\begin{center}
\begin{tikzpicture}[]
\node (g) at (-0.5, -3.5) {$\gamma$};
\node (r) at (-0.5, -2.5) {$\subselmg{\Ran}$};
\node (conv) at (-0.5, -1.5) {$\subselmg{\conv}$};
\node (f2) at (-0.5, -0.5) {$\subselmg{\fin^2}$};
\node (fbeta) at (-0.5, 0.5) {$\subselmg{\fin^\beta}$};
\node (falpha) at (-0.5, 1.5) {$\subselmg{\fin^\alpha}$};
\node (deltap) at (-0.5, 2.5) {$\delta'$};
\node (delta) at (-0.5, 3.5) {$\delta$};
\node (pi) at (-0.5, 4.5) {$\pi$};
\node (s) at (-3.5, -2) {$\subselmg{\calS}$};
\node (nwd) at (-3.5, -0.5) {$\subselmg{\nwd}$};
\node (ed) at (1.5, -2.3) {$\subselmg{\ED}$};
\node (edfin) at (3.3, -2.1) {$\subselmg{\EDfin}$};
\node (sum) at (4.5, -1) {$\subselmg{\summable}$};
\node (z) at (2.5, -0.5) {$\subselmg{\Z}$};
\foreach \from/\to in {conv/z} \draw [->] (\from) -- (\to);
\foreach \from/\to in {ed/f2} \draw [line width=.15cm,white] (\from) -- (\to);
\foreach \from/\to in {g/r,g/s,r/conv,conv/f2,f2/fbeta,fbeta/falpha,falpha/deltap,deltap/delta,delta/pi,conv/nwd,s/nwd,nwd/pi,r/ed,ed/f2,ed/edfin,edfin/sum,sum/z} \draw [->] (\from) -- (\to);
\end{tikzpicture}
\end{center}
\caption{Known relations among investigated notions for reals. Here $2\leq\beta<\alpha<\omega_1$.}
\label{diagram:katetov}
\end{figure}
\par
We say that a~topological space~$X$ has property~$\delta'$, if there is $\alpha<\omega_1$ such that for any open $\omega$-cover~$\V$ we have $X\in\rmL^\alpha(\V)$. Thus if $X$ has property~$\delta'$, then $X$ has property~$\delta$, and we do not know whether the~reversed implication does hold.
\begin{corol}[\CH]\label{finalpha-nonfinbeta}
The~statement 
\begin{center}
    ``There is an~$\subselmgooR{\I}$-space $A\subseteq\PP(\omega)$ that is not an~$\subselmgR{\J}$-space."
\end{center}
is true for the~following ideals.
\begin{enumerate}[(1)]
    \item $\I\not\leqk\fin^\beta$, $\J=\fin^\beta$, for $0<\beta<\omega_1$. In particular, $\I$ may be $\nwd$, $\calS$, and $\fin^\alpha$ for $\beta<\alpha<\omega_1$.
    \item $\I=\fin^\alpha$, $\J$ a~$\rmP^+$-ideal, for $1<\alpha<\omega_1$. In particular, $\J$ may be $\fin, \summable, \ED, \EDfin, \Ran, \calS$.
    \item $\I=\nwd$, $\J$ an~$\Fsigma$ ideal. In particular, $\J$ may be $\fin, \summable, \ED, \EDfin, \Ran, \calS$.  
    \item $\I\not\leqk\finfin$ and $\J\leqk\conv$. In particular, $\I$ may be $\nwd$, $\calS$, and $\J$ may be $\conv$, $\Ran$. 
    \item $\conv\leqk\I\not\leqk\J$, and $\J$ is $\Fsigma$ ideal. In particular, $\I$ may be $\conv$, $\fin^\alpha$ for $\alpha>1$, and $\J$ may be $\calS$, $\Ran$, $\ED$, $\EDfin$, $\summable$.      
    \item $\I=\summable$, $\J=\EDfin$.
\end{enumerate}
Moreover, there is $A\subseteq\PP(\omega)$ possessing property~$\delta'$ (and so $\delta$ as well) that is not an~$\subselmgR{\fin^\beta}$-space.
\end{corol}
\begin{proof}
    Each item follows by the~corresponding item in~Lemma~\ref{katetov-rest} in combination with Theorem~\ref{subseqn-nonsubseqn}. We include a~sample proof of item~(5). Indeed, let us take ideals $\I,\J$ such that $\conv\leqk\I\not\leqk\J$, and $\J$ is $\Fsigma$ ideal. By item~(5) in~Lemma~\ref{katetov-rest} we have $\I\not\leqcountk\J$. Finally, by Theorem~\ref{subseqn-nonsubseqn} there exists an~$\subselmgooR{\I}$-space $A\subseteq\PP(\omega)$ that is not an~$\subselmgR{\J}$-space
\end{proof}


\section{Ideal Fr\' echet--Urysohn property}\label{S-frecheturysohn}

According to~\cite{BorFar,Su20}, $\cp{X}$ possesses $\J$-Fr\' echet--Urysohn property, shortly~$\zsubselmgR{\J}$,\footnote{We shall write $\zsubselmg{\J}$ as well. Symbol $\Omega_\zero$ in the~notation stands for the~family of all $A\subseteq\cp{X}\setminus\{\zero\}$ with $\zero\in\overline{A}$, and $\Gamma_\zero$ is the~family of all sequences convergent to~$\zero$. Note that $\cp{X}$ is a~topological group.} if for every $A\subseteq\cp{X}$ and $f\in\overline{A}$ there is a~sequence in~$A$ that $\J$-converges to~$f$, which means that there exists a sequence  $\langle a_n:n\in\omega\rangle$ of elements of $A$ such that
$\{n\in\omega:a_n\not\in O\}\in \J$ for any open neighbourhood~$O$ of~$f$. Note that by Theorem~\ref{epsilon-char}, if $X$ is an~$\varepsilon$-space, then it is enough to consider countable $A\subseteq\cp{X}$.
 We summarize the~relations
 between properties $\zsubselmgR{\J}$
 for various standard critical ideals $\J$ in~Diagram~\ref{diagram:katetovfu}, which is similar to Diagram~\ref{diagram:katetov}. 
\begin{figure}[H]
\begin{center}
\begin{tikzpicture}[]
\node (g) at (-0.5, -3.5) {$\fu$};
\node (r) at (-0.5, -2.5) {$\zsubselmg{\Ran}$};
\node (conv) at (-0.5, -1.5) {$\zsubselmg{\conv}$};
\node (f2) at (-0.5, -0.5) {$\zsubselmg{\fin^2}$};
\node (fbeta) at (-0.5, 0.5) {$\zsubselmg{\fin^\beta}$};
\node (falpha) at (-0.5, 1.5) {$\zsubselmg{\fin^\alpha}$};
\node (delta) at (-0.5, 2.5) {$\partlim$};
\node (pi) at (-0.5, 4.5) {$\pytkeev$};
\node (s) at (-3.5, -2) {$\zsubselmg{\calS}$};
\node (nwd) at (-3.5, -0.5) {$\zsubselmg{\nwd}$};
\node (ed) at (1.5, -2.3) {$\zsubselmg{\ED}$};
\node (edfin) at (3.3, -2.1) {$\zsubselmg{\EDfin}$};
\node (sum) at (4.5, -1) {$\zsubselmg{\summable}$};
\node (z) at (2.5, -0.5) {$\zsubselmg{\Z}$};
\foreach \from/\to in {conv/z} \draw [->] (\from) -- (\to);
\foreach \from/\to in {ed/f2} \draw [line width=.15cm,white] (\from) -- (\to);
\foreach \from/\to in {g/r,g/s,r/conv,conv/f2,f2/fbeta,fbeta/falpha,falpha/delta,delta/pi,conv/nwd,s/nwd,nwd/pi,r/ed,ed/f2,ed/edfin,edfin/sum,sum/z} \draw [->] (\from) -- (\to);
\end{tikzpicture}
\end{center}
\caption{Ideal Fr\' echet--Urysohn property of $\cp{X}$ for $X\subseteq\R$. Ordinal~$\beta$ is less than~$\alpha$.}
\label{diagram:katetovfu}
\end{figure}
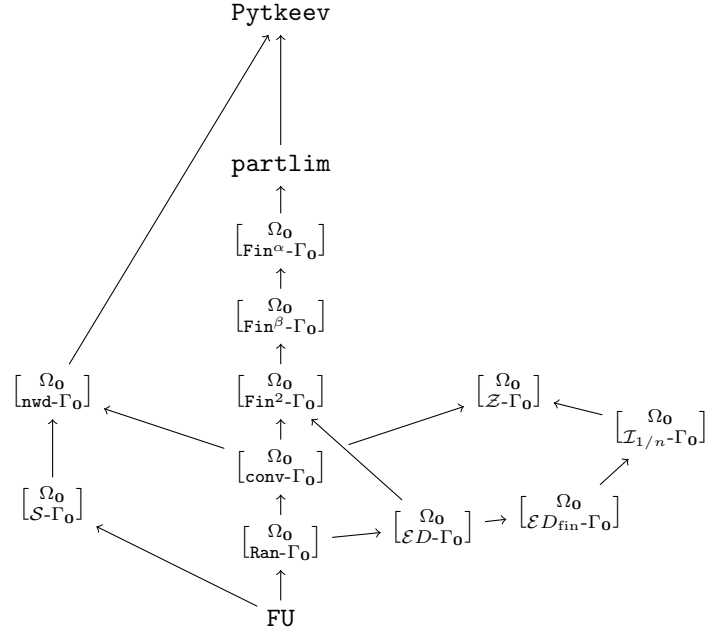
\par
 
Diagrams~\ref{diagram:katetov}, \ref{diagram:katetovfu} raise from Diagram~\ref{diagram:katetovideals} depicting the  Kat\v etov order on these ideals. Here we use the fact that if $\I_1\leqk\I_2$, then any~$\zsubselmgR{\I_1}$-space is an~$\zsubselmgR{\I_2}$-space as well. We recall that
  \cite[Theorem~4.5]{OrTs13} gives a characterization of $\delta$-spaces by a property of $\cp{X}$
  phrased in terms of pointwise convergence
  of partial real-valued functions on $X$. In what follows we shall refer to the latter property of $\cp{X}$ as to $\partlim$. We do not define here this property, which is exactly
   stated in \cite[Theorem~4.5(1)]{OrTs13}, because we do not work with it directly in the paper.
\par
The~main goal of the~present section is to prove that consistently many  arrows 
``missing'' in~Diagram~\ref{diagram:katetovfu} actually cannot be added,
by relating these properties of $\cp{X}$ to the covering properties of $X$ studied earlier in this 
paper. 
If $\cp{X}$ is an~$\zsubselmgR{\J}$-space, then $X$ is an~$\subselmgR{\J}$-space, see~\cite{Su20}. Hence, for sets of reals, we obtain the~following diagram of implications, see Proposition~\ref{subsequenceimpliespi}. 
\begin{figure}[H]
\begin{center}
\begin{tikzpicture}[]
\node (a) at (-5, -3.5) {$\gamma$};
\node (as) at (-5, -1) {$\fu$};
\node (aa) at (-2, -3.5) {$\subselmgR{\fin^\alpha}$};
\node (aas) at (-2, -1) {$\zsubselmgR{\fin^\alpha}$};
\node (c) at (1, -3.5) {$\delta$};
\node (cs) at (1, -1) {$\partlim$};
\node (e) at (4, -3.5) {$\pi$};
\node (es) at (4, -1) {$\pytkeev$};
\node (f) at (-8, -3.5) {$X$};
\node (fs) at (-8, -1) {$\cp{X}$};
\foreach \from/\to in {aas/aa, c/e, a/aa, aa/c, as/aas, aas/cs, cs/es} \draw [->] (\from) -- (\to);
\foreach \from/\to in {as/a, cs/c, es/e} \draw [<->] (\from) -- (\to);
\end{tikzpicture}
\end{center}
\caption{Relations among investigated properties. $\alpha$ is a~countable ordinal.}
\label{diagram:1}
\end{figure}
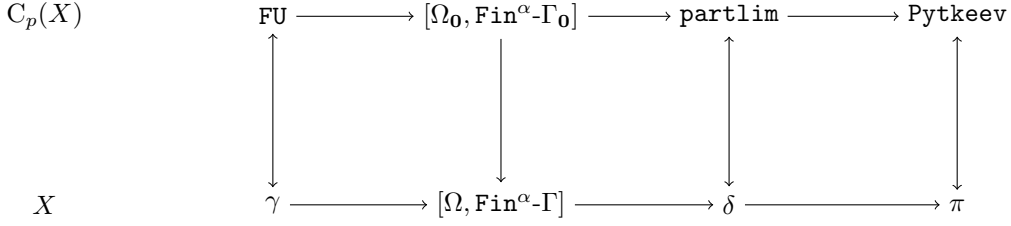
It is clear that if $\J$ is not tall and $X$ is an~$\subselmgR{\J}$-space,
 then $\cp{X}$ is an~$\zsubselmgR{\J}$-space:
 This happens because the former property
 is equivalent to being a $\gamma$-space, and the latter to
being  a Fr\'echet--Urysohn one.
  We show that this implication holds for arbitrary (not necessary tall) ideals in the realm of
   $\sonemm$-spaces. Let us recall that a~topological space~$X$ is an~$\sonemm$-space if for every sequence $\seqn{\V_n}{n}$ of open $\omega$-covers of~$X$ there is $V_n\in\V_n$ such that $\set{V_n}{n\in\omega}$ is an~$\omega$-cover of~$X$, see~\cite{Comb1} and references therein for more information on such spaces.
\begin{prop}\label{subselplusmm}
 If $X$ is an~$\sonemm$- and ~$\subselmgR{\J}$-space,  then $\cp{X}$ is an~$\zsubselmgR{\J}$-space.
\end{prop}
\begin{proof}
We assume that we have continuous functions $f_m\colon X\to[0,1]$ such that $\zero\in\overline{\set{f_m}{m\in\omega}}$. Then for an~injective sequence~$\seqn{x_m}{m}$ in~$X$ and a~nonnegative integer~$n$ define $V_{n,m}=\set{x\in X}{f_m(x)<2^{-n}\wedge x\neq x_m}$. One can check that each sequence $\seqn{V_{n,m}}{m}$ is an~$\omega$-cover of~$X$. By the~assumption, there is $\varphi\colon\omega\to\omega$ such that $\seqn{V_{n,\varphi(n)}}{n}$ is an~$\omega$-cover of~$X$. Since $X$ is an~$\subselmgR{\J}$-space, there is a $\J$-to-one function $\psi\colon\omega\to\omega$ such that $\seqn{V_{\psi(n),\varphi(\psi(n))}}{n}$ is a~$\J$-$\gamma$-cover of~$X$.
\par
We claim that the~sequence $\seqn{f_{\varphi(\psi(n))}}{n}$ is $\J$-$\gamma$-converging to~$\zero$. Indeed, let $\varepsilon>0$ and $x\in X$. Note that $\set{n}{2^{-\psi(n)}<\varepsilon}\in\J^\ast$. Also, $\set{n}{f_{\varphi(\psi(n))}(x)<\varepsilon}$ is a superset of the~intersection of the~sets $\set{n}{2^{-\psi(n)}<\varepsilon}$ and $\set{n}{x\in V_{\psi(n),\varphi(\psi(n))}}$, both being from~$\J^\ast$.
\end{proof}
\par
We shall use Proposition~\ref{subselplusmm} in constructions of $\zsubselmgR{\J}$-spaces in~Theorem~\ref{zsubseqn-nond}. Thus an~appropriate sufficient condition for a~set of reals to be an~$\sonemm$-space is needed. The~following theorem is a~consequence of results in~\cite{KoSch03,TsWe04}. Let us recall that a~topological space~$X$ has the~\emph{Rothberger property} if for every sequence $\seqn{\V_n}{n}$ of open covers of~$X$ there are $V_n\in\V_n$ such that $\set{V_n}{n\in\omega}$ is an~open cover of~$X$. A~topological space~$X$ has the~\emph{Hurewicz property} if for every sequence $\seqn{\V_n}{n}$ of open covers of~$X$ without finite subcovers there are finite $\W_n\subseteq\V_n$ such that $\set{\bigcup\W_n}{n\in\omega}$ is a~$\gamma$-cover of~$X$. For more on Rothberger and Hurewicz properties see~\cite{buk-str,Comb1}.
\begin{theorem}[L.D.R.~Ko\v cinac--M.~Scheepers, B.~Tsaban--T.~Weiss]\label{H+R-sonemm}
Let $X$ be a~set of reals such that $X^n$ has the~Hurewicz property for each~$n$. If $X$ has the~Rothberger property, then $X$ is an~$\sonemm$-space.
\end{theorem}
\par
Theorem~\ref{H+R-sonemm} will help us to conclude that the spaces we construct below with the help of
$\mathcal J$-evading double scales (see p.~\pageref{evades} for the definition thereof) satisfy $\sonemm$, thus
allowing to use Proposition~\ref{subselplusmm} in order to establish the corresponding local properties of function spaces.

\begin{lema}\label{quasiordering-families}
Let $\J$ be an ideal on $\omega$ and suppose that $\set{a_\xi}{\xi<\bb}\subseteq\infin$  $\J$-evades
$\set{ b_\xi}{\xi<\bb}\subseteq\infin$
 such that for every $f\in\infin$ there exists $\beta<\bb$ with $b_\beta\subseteq f$.
         Then
for any $f$ there exists  $\beta<\bb$ such that the~family $\set{\set{i}{f(i)\leq a_\alpha(i)}}{\alpha\geq\beta}$ has \fip.
     Moreover, if $\J$ is meager,  then the~filter induced by the~aforementioned family is meager as well.
\end{lema}
\begin{proof}
Set $f'(0)=f(1)$ and $f'(n+1)=f(f'(n)+1)$ for all $n\in\omega$. 
\begin{clm*}
If  $b\subseteq f'$,  then   
 for any $a\in\infin$ and $k\in\omega$, if $a\cap [b(k), b(k+1))=\emptyset$,  then 
$f(b(k))< a(b(k))$, i.e., $b(k)\in\{m:f(m)< a(m)\}$.
\end{clm*}
\begin{proof}
Given any $k\in\omega$ let us denote by $i$ and $j$ natural numbers such that $b(k)=f'(i)$ and $b(k+1)=f'(j)$.
Then $b(k+1)=f'(j)\geq f'(i+1)>f(f'(i))=f(b(k))$ and $a(b(k))\geq b(k+1)$ since $a\cap [b(k),b(k+1))=\emptyset$.
Combining these two inequalities we get $a(b(k))>f(b(k))$.
\end{proof}
Let $\beta\in\bb$ be such that $b_\beta\subseteq f'$.
 Since $\set{a_\xi}{\xi<\bb}\subseteq\infin$  $\J$-evades
$\set{ b_\xi}{\xi<\bb}\subseteq\infin$, we have
$$ F_{\beta,\alpha}:=\{ k\in\omega:[b_\beta(k), b_\beta(k+1))\cap a_\alpha=\emptyset\}\in\J^* $$
for all $\alpha\geq \beta$. Applying the claim proved above we conclude that 
$b_\beta[F_{\beta,\alpha}]\subseteq \{i:\ f(i)<a_{\alpha}(i)\}$ for all $\alpha\geq\beta$, 
and hence the family 
$\G:=\set{\set{i}{f(i)\leq a_\alpha(i)}}{\alpha\geq\beta}$ has \fip.
Moreover, if $\J$ is meager, so is the filter generated by $\{b_\beta[F]:F\in\J^*\}$.
Since $\G$ is a subset of the latter one, it is meager as well.
\end{proof}
\par
Our assumption in Lemma~\ref{quasiordering-families} implicitly gives
$\bb=\cc$: Since there are almost disjoint subfamilies of 
$[\omega]^\omega$ of size $\cc$, no family $B\subseteq [\omega]^\omega$
of size $<\cc$ can have the property that for each $f\in [\omega]^\omega$ there exists $b\in B$ with $b\subseteq f$. This is not an essential restriction in this context since we need Lemma~\ref{quasiordering-families} for the proof of 
Theorem~\ref{zsubseqn-nond} where we anyway need to assume CH due to the usage
of $\leq^\omega_K$.
\smallskip

 The~proof of the next statement is a direct application of \cite[Lemma~13]{BaTs} in the style of 
\cite[Theorem~5.1]{TsZd08}. We shall construct a~variation of a~scale
satisfying the premises of Proposition~\ref{subselplusmm}, 
 see~\cite{buk-str,Ts11,TsZd08} for an~overview of similar results. Let us recall that $A\subseteq\PP(\omega)$ is $\kappa$-concentrated on $\fin$ if $|A\setminus V|<\kappa$ for any open set $V\supseteq\fin$.  
\begin{prop}\label{representativefamiliesfunc}
Under the assumptions of Lemma~\ref{quasiordering-families}, 
the set  $A=\set{a_\alpha}{\alpha<\bb}\cup\fin$
is $\bb$-concentrated on~$\fin$ and has 
 all finite powers Hurewicz. Moreover, 
if $\bb=\mathrm{cov}(\mathcal M)$
and $\set{a_\alpha, b_\alpha}{\alpha<\bb}$ is $\J$-evading double scale, then
 $\cp{A}$ is an~$\zsubselmgR{\J}$-space.
\end{prop}
\begin{proof}
  By induction on $k\in\omega$, we shall show first that $A^k$ has the~Hurewicz property. Suppose first that $k=1$
and consider a continuous $\Psi\colon A\to\infin$. By  Hurewicz' characterization \cite{Hur27}
(see also  \cite[Theorem 4.4]{COC2} or \cite{rec} for more recent proofs) it is enough to show that the~image~$\Psi[A]$ is bounded in~($\baire\omega$, $\leq^\ast$). 
By \cite[Lemma~13]{BaTs}, there is $g\in\infin$ such that for any $a\in A$ we have
\[
\set{i}{g(i)<a(i)}\subseteq\set{i}{\Psi(a)(i)\leq g(i)}.
\]
By Lemma~\ref{quasiordering-families}, there exist $\beta<\bb$ and  a~meager filter~$\F$ such that $\set{\set{i}{g(i)<a_\alpha(i)}}{\alpha\geq\beta}\subseteq\F$. Then $g<_\F a_\alpha$ for $\alpha\geq\beta$, and by the~above inclusion, we obtain that $\Psi(a_\alpha)\leq_\F g$ for $\alpha\geq\beta$. Since $\Psi[\set{a_\alpha}{\alpha<\beta}\cup\fin]$ has  cardinality $<\bb$, the~image $\Psi[A]$ is bounded in~($\baire\omega$, $\leq_\F$). However, $\F$ is meager and hence $\Psi[A]$ is bounded in~($\baire\omega$, $\leq^\ast$) as well, see
\cite[Theorem~5.3]{farksouk} or \cite[Lemma~2.26]{TsZd08}  for a proof of this easy fact.
\par
Now suppose that $A^{k-1}$ is Hurewicz for some $k\geq 2$ and consider a continuous
 $\Psi\colon A^k\to\infin$. By \cite[Lemma~13]{BaTs}, there is $g\in\infin$ such that for any $a^0,\dots, a^{k-1}\in A$ we have
\[
\set{i}{g(i)<\min\{a^0(i),\dots,a^{k-1}(i)\}}\subseteq\set{i}{\Psi(a^0,\dots,a^{k-1})(i)\leq g(i)}.
\]
By Lemma~\ref{quasiordering-families}, there is $\beta<\bb$ and a~meager filter~$\F$ such that $\set{\set{i}{g(i)<a_\alpha(i)}}{\alpha\geq\beta}\subseteq\F$. For any $\alpha_0,\dots,\alpha_{k-1}\geq\beta$ we obtain that $\set{i}{g(i)<\min\{a_{\alpha_0}(i),\dots,a_{\alpha_{k-1}}(i)\}}\in\F$. It follows from the above that $\Psi(a_{\alpha_0},\dots,a_{\alpha_{k-1}})\leq_\F g$ for any $\alpha_0,\dots,\alpha_{k-1}\geq\beta$, and hence also $\Psi[E]$
is bounded in~($\baire\omega$, $\leq^\ast$) for $E=\set{
\langle a_{\alpha_0},\dots,a_{\alpha_{k-1}}\rangle}{\alpha_0,\dots,\alpha_{k-1}\geq\beta}$  because $\F$ is meager, see the discussion in case $k=1$.
\par
Let us show that the~set~$\Psi[A^k\setminus E]$ is bounded. For fixed $a\in\set{a_\xi}{\xi<\beta}$ and $i<k$ consider the~set $E^i_a=\set{\langle a_{\alpha_0},\dots,a_{\alpha_{i-1}},a,a_{\alpha_{i+1}},\dots,a_{\alpha_{k-1}}\rangle}{\alpha_0,\dots,\alpha_{i-1},
\alpha_{i+1},\dots,\alpha_{k-1}\geq\beta}$. There are fewer than $\bb$ of sets~$E^i_a$, they cover $A^k\setminus E$,  and each $E^i_a$ is homeomorphic to~$A^{k-1}$. It suffices now to apply our inductive assumption to conclude that $\Psi[A^k]$ is bounded in $(\baire\omega$, $\leq^\ast)$, which completes our proof of all finite powers of $A$ being Hurewicz.
\par
Given any open neighbourhood $O$ of $\fin$,
the complement  $\PP(\omega)\setminus O$  is a compact subspace of $\infin$, hence it is  bounded by some $f\in\infin$, and therefore there exists $\beta<\bb$ such that $O\ni a_\alpha$ for all $\alpha\geq \beta$,
see Lemma~\ref{quasiordering-families}. This means that $A$ is $\bb$-concentrated.
\par
Now if $\bb=\mathrm{cov}(\mathcal M)$, then $A$ is  $\mathrm{cov}(\mathcal M)$-concentrated and
hence Rothberger by \cite[Theorem 4.2]{COC2}. Thus by  Theorem~\ref{H+R-sonemm} we conclude that $A$ is an~$\sonemm$-space.
If in addition $\set{a_\alpha, b_\alpha}{\alpha<\bb}$ is $\J$-evading double scale,
 Lemma~\ref{representativefamilies} implies that $A$ is an~$\subselmgR{\J}$-space.  Hence, by Proposition~\ref{subselplusmm}, we conclude that $\cp{A}$ is an~$\zsubselmgR{\J}$-space, which completes our proof.
\end{proof}
\begin{theorem}[\CH]\label{zsubseqn-nond}
Let  $\I$ be  a~meager ideal.
\begin{enumerate}[(1)]
    \item If $\I\not\leqcountk\fin^\alpha$ for each $\alpha<\omega_1$, then there is $A\subseteq\PP(\omega)$ such that $\cp{A}$ is an~$\zsubselmgR{\I}$-space but $A$ is not a~$\delta$-set.
    \item If $\J$ is an~ideal and $\I\not\leqcountk\J$, then there is $A\subseteq\PP(\omega)$ such that $\cp{A}$ is an~$\zsubselmgR{\I}$-space but not an~$\zsubselmgR{\J}$-space.
\end{enumerate}
\end{theorem}
\begin{proof}
It essentially suffices to  repeat the  proof of Theorem~\ref{subseqn-nonsubseqn}, with the following addition: When constructing the family $\set{b_\xi}{\xi<\cc}$, we have to make sure that it satisfies the premises of Lemma~\ref{quasiordering-families},
i.e., that for every $f\in\infin$ there exists $\xi$ with $b_\xi\subseteq f$. That is possible by enumerating all infinite subsets of $\omega$ and asking in the~$\xi$-th step that $b_\xi$ is a subset of $\xi$-th set, see Lemma~\ref{GMlemma}(1).
 Furthermore,  we have to use Proposition~\ref{representativefamiliesfunc} instead of Lemma~\ref{representativefamilies}.
\end{proof}
\par
As a consequence we get the~following   counterpart of Corollary~\ref{finalpha-nonfinbeta} for spaces of functions, which
gives an extensive list of arrows which consistently cannot be added to Diagram~\ref{diagram:katetovfu}.
\begin{corol}[\CH]\label{9_6_n}
The~statement
\begin{center}
    ``There is $A\subseteq\PP(\omega)$ such that $\cp{A}$ is an~$\zsubselmgR{\I}$-space but not an~$\zsubselmgR{\J}$-space"
\end{center}
is true for all pairs of ideals listed in Corollary~\ref{finalpha-nonfinbeta}.
\end{corol}


\section{Consequences for cardinal invariants and other covering notions }\label{S-cardinals}

P.~Borodulin--Nadzieja and B.~Farkas~\cite{BorFar} have introduced and investigated cardinal invariant $\ppsq{\J}$ defined for $\square\in\{\text{\rm 1-1}, \mathrm{KB}, \mathrm{K}\}$ by 
\begin{center}
$\ppsq{\J}=\min\set{|\A|}{\A\subseteq\PP(\omega)\text{ has }\fup\wedge \A\not\leq_\square\J}$.
\end{center}
It has been shown in~\cite[Proposition~3.3]{Su20} that $\non{\subselmgsqR{\J}}=\ppsq{\J}$.\footnote{$\ppsq{\J}$ is a~uniformity number of $\zsubselmgsqR{\J}$ as well, see~\cite[Proposition~9.4]{Su20} or~\cite[Theorem~3.2]{BorFar} for details.} Moreover, it has been shown that $\ppsq{\Ran}=\ppsq{\conv}=\pp$. Since $\non{\gamma}=\pp$ by F.~Galvin and A.W.~Miller~\cite{GM}, and $\non{\pi}=\pp$ by P.~Simon and B.~Tsaban~\cite{SimTsa}, we obtain the~following consequence of Lemma~\ref{subsequenceimpliespi} and Proposition~\ref{subsequenceimpliespidelta}.
\begin{corol}\label{pseudoin-number}
\begin{enumerate}[(1)]
    \item If $\J$ is not $\omega$-hitting, then $\ppkb{\J}=\pp$.
    \item If $\J\leqk\nwd$ or $\J\leqk\fin^\alpha$, then $\ppk{\J}=\pp$.
    \item $\ppk{\fin\times\fin}=\ppk{\fin^\alpha}=\ppk{\ED}=\ppk{\calS}=\ppk{\nwd}=\pp$.
\end{enumerate}
\end{corol}
\par
Another consequence is related to our Lemma~\ref{nothavingproperties} and $\omega$-hitting families of reals. In fact, the~main essence of the~proof is the~above mentioned equality $\non{\pi}=\pp$ by P.~Simon and B.~Tsaban~\cite{SimTsa}.
\begin{corol}\label{minimalhitting}
$\pp=\min\set{|A|}{A\subseteq\PP(\omega)\text{ has \fup, }A\text{ is }\omega\text{-hitting}}$.
\end{corol}
\begin{proof}
Let us denote the~right-hand side cardinality by~$\kappa$. If $A\subseteq\PP(\omega)$ has \fup, and $|A|<\pp$, then any infinite pseudointersection of $A^\ast=\set{\omega\setminus a}{a\in A}$ is a~witness for $A$ being not \tall, no word about $\omega$-hitting. Thus $\kappa\geq\pp$.
\par
To show the~reversed inequality, let $X\subseteq\PP(\omega)$ be a~non-$\pi$-set of size~$\pp$ with a~clopen $\omega$-cover $\V=\set{V_n}{n\in\omega}$ witnessing this (enumerated bijectively). For each finite set~$F\subseteq X$ we set $a_F=\set{n}{F\not\subseteq V_n}$. Then $A=\set{a_F}{F\in[X]^{<\omega}}$ has \fup\ (because $\V$ is an~$\omega$-cover), and is $\omega$-hitting since $\V$ is a~witness of $X$ being not a~$\pi$-set.
\end{proof}
\par
In contrast to $\subselmgR{\J}$-space whose definition involves all $\omega$-covers, a~similar notion related just to $\I$-$\gamma$-covers for an~ideal~$\I$ is studied as well. Let $\square\in\{\text{\rm 1-1}, \mathrm{KB}, \mathrm{K}\}$. We say that $X$ is an~$\subselggsqR{\I}{\J}$-space if for every $\I$-$\gamma$-cover $\seqn{V_n}{n}$ of~$X$ there is $\square$-function $\varphi\in\baire\omega$ such that $\seqn{V_{\varphi(m)}}{m}$ is a~$\J$-$\gamma$-cover of~$X$, see~\cite{SoSu,Su20}. Any $\subselmgR{\J}$-space is an~$\subselggR{\I}{\J}$-space, the~relation is depicted in Diagram~\ref{diagram:ijsubseq}. Note that any topological space~$X$ is both an~$\subselmgR{\J}$-space and a~$\subselggfR{\J}$-space if only if $X$ is a~$\gamma$-set. 
\begin{figure}[H]
\begin{center}
\begin{tikzpicture}[]
\node (b) at (0.5, -5) {$\gamma$};
\node (c) at (4, -5) {$\subselmg{\J}$};
\node (d) at (0.5, -2) {$\subselggf{\I}$};
\node (e) at (4, -2) {$\subselgg{\I}{\J}$};
\foreach \from/\to in {b/d, b/c, d/e, c/e} \draw [line width=.15cm,
white] (\from) -- (\to);
\foreach \from/\to in {b/d, b/c, d/e, c/e} \draw [->] (\from) -- (\to);
\end{tikzpicture}
\end{center}
\caption{$\subselggR{\I}{\J}$-space.}
\label{diagram:ijsubseq}
\end{figure}
\par
Properties $\subselmgR{\I}$ and $\subselggfR{\I}$ are independent. The~uniformity number of~$\subselggfR{\I}$-space is $\covh{\I}$, see~\cite{SoSu}. Thus if $\covh{\I}>\pp$, then there is trivially an~$\subselggfR{\I}$-space that is not a~$\gamma$-set, and so not an~$\subselmgR{\I}$-space (by the~latter paragraph).  On the~other hand, we show below that the~reversed implication is also not provable.
\begin{lema}\label{gfinnotlessiu}
Let $\G$ be an~ultrafilter, and $\J$ a~Borel ideal. If $\I\not\leqcountk\J$, then for any $\J$-to-one function $f\in\baire\omega$ and countable $B\subseteq\infin$ containing~$\omega$, there is an~infinite set $a\in\G^\ast\cap\I$ such that $f^{-1}[a]\in\J^+$ and $\bi_b[a]\in\I$ for any $b\in B$.
\end{lema}
\begin{proof}
Let $f\in\baire\omega$ be $\J$-to-one, $B\subseteq\infin$, $|B|\leq\omega$. By $\I\not\leqcountk\J$, there is $a'\in\infin$ such that $f^{-1}[a']\in\J^+$ and $\bi_b[a']\in\I$ for any $b\in B$. We claim that there is $a\subseteq a'$ such that $a\in\G^\ast$ and $f^{-1}[a]\in\J^+$. Indeed, otherwise $\G^\ast\restriction a'\leq_f\J$ that is not possible for a~Borel ideal~$\J$ by Lemma~\ref{Gfin-notplus}. It follows directly that $\bi_b[a]\in\I$ for any $b\in B$. Finally, if $\omega\in B$ then $a=\bi_\omega[a]\in\I$.
\end{proof}
\begin{corol}
\begin{enumerate}[(1)]
    \item {\rm($\pp=\cc$)} If $\I,\J$ are \tall, then there is an~$\subselmgooR{\J}$-space $A\subseteq\PP(\omega)$ that is not an~$\subselggfR{\I}$-space.
    \item {\rm($\CH$)} If $\I,\J$ are $\omega$-hitting, then there is an~$\subselmgooR{\J}$-space $A\subseteq\PP(\omega)$ that is neither an~$\subselggfR{\I}$-space nor a~$\pi$-set. 
    \item {\rm($\CH$)} Let $\G$ be an~ultrafilter, and $\J$ a~Borel ideal such that $\I\not\leqcountk\J$. Then there is an~$\subselmgooR{\I}$-space $A\subseteq\PP(\omega)$ that is not an~$\subselggR{(\I\cap\G^\ast)}{\J}$-space.
\end{enumerate}
\end{corol}
\begin{proof}
(1) A~consequence of a~proof of Theorem~\ref{subselection-nongamma} since \tall\ ideal~$\I$ in its~proof was selected in~advance, and the~constructed set $A$ is \tall\ and contained in~$\I$. By~\cite[Lemma~3.2(1)]{Su20}, no~$A\not\leqk\fin$ contained in~$\I$ is an~$\subselggfR{\I}$-space. 
\par
(2) A~consequence of a~proof of $\text{\rm (b)}\to\text{\rm (a)}$ in Theorem~\ref{subsequence-nonpi}. There is constructed an~$\omega$-hitting set $A\subseteq\J$ that is an~$\subselmgooR{\I}$-space, but not a~$\pi$-set. However, any $\omega$-hitting set is \tall, thus $A\subseteq\J$ is \tall. By~\cite[Lemma~3.2(1)]{Su20}, $A$ is not a~$\subselggfR{\J}$-space. To obtain the~statement of our corollary it is enough to interchange the~role of~$\I$ and~$\J$.
\par
(3) A~consequence of a~proof of Theorem~\ref{subseqn-nonsubseqn}: The~set~$A$ constructed there is contained in~$\I\cap\G^
\ast$ (once using Lemma~\ref{gfinnotlessiu} instead of Lemma~\ref{gfinnotlessi}(2)), and $A\not\leqk\J$. Hence, by~\cite[Lemma~3.2(1)]{Su20}, $A$ is not an~$\subselggR{(\I\cap\G^
\ast)}{\J}$-space.
\end{proof}
\par
Finally, let us point out that $\ppsq{\J}$ is a~uniformity number of $\zsubselmgsqR{\J}$ as well, see~\cite[Proposition~9.4]{Su20} or~\cite[Theorem~3.2]{BorFar} for details. Moreover, a~functional version of an~$\subselggsqR{\I}{\J}$-space is investigated in~\cite{rep,rep2,SoSu,Su20}, including a~cardinal invariant~$\ppsq{\I,\J}$ as the~uniformity number of an~$\subselggsqR{\I}{\J}$-space, and its functional version.


\section{Questions}

We have shown in Corollary~\ref{finalpha-nonfinbeta} that many of the~implications in Diagram~\ref{diagram:katetov} cannot be provably reversed. However, our reasoning was based on the~fact that for most of critical ideals~$\I,\J$ in Diagram~\ref{diagram:katetovideals}, if $\I\not\leqk\J$, then $\I\not\leqk\J\restriction J$ for each $\J$-positive set~$J$. Since the~ideal $\finfin$ is Kat\v etov below $\conv\restriction J$ for many $\conv$-positive sets~$J$, the~following remains open.
\begin{quest}\label{quest-counterexamples}
Is it consistent that there is an~$\subselmgR{(\finfin)}$-set that is not an~$\subselmgR{\conv}$-set?
\end{quest}
Question~\ref{quest-counterexamples} applies also to the~other pairs. We do not know whether it is consistent that the~following holds: $\subselmgR{\fin^\alpha}\not\to\subselmgR{\nwd}$, $\subselmgR{\ED}\not\to\subselmgR{\conv}$. 
\par
A~lot of attention in Sections~\ref{S-pytkeev_nondelta}, \ref{S-gamma_sets} is devoted to  ideals~$\I,\J$ such that $\I\leqcountk\J$ yields $\I\leqrestk\J$. However, we do not know a~general combinatorial description of such pairs.
\begin{quest}\label{katetov-relations}
Which assumption on $\I,\J$ guarantee that $\I\leqcountk\J$ implies $\I\leqrestk\J$ (or even $\I\leqk\J$)?
\end{quest}
Our constructions in Sections~\ref{S-pytkeev_nondelta}, \ref{S-gamma_sets} and~\ref{S-frecheturysohn} needed~\CH\ since they rely on the~relation $\leqcountk$ involving countable families. The~employed transfinite inductions are therefore of the~length~$\omega_1$. We do not know whether this restriction could be weakened.
\begin{quest}
Could the~assumption of~\CH\ in Theorems~\ref{subseqn-nond}, \ref{subseqn-nonsubseqn} and~\ref{zsubseqn-nond} be weakened to {\rm\bf MA} or $\pp=\cc$?
\end{quest}
In the~constructions of our examples we always assume $\pp=\cc$. However, P.~Borodulin-Nadzieja and B.~Farkas~\cite{BorFar} have shown that it is consistent that there is an~ideal~$\J$ with the~uniformity number~$\ppk{\J}$ of an~$\subselmgR{\J}$-space larger than the~pseudointersection number~$\pp$.
\begin{quest}
Is $\ppsq{\J}=\cc$ enough to guarantee the~existence of an~$\subselmgsqR{\J}$-space of cardinality~$\cc$?
\end{quest}
It has been also proven in~\cite{BorFar} that in the~Cohen real model there exists a~meager ideal~$\J$ such that $\ppk{\J}>\pp$. On the~other hand, our results (Corollary~\ref{pseudoin-number}) together with those in~\cite{Su20} show that most of the~critical ideals in Diagram~\ref{diagram:katetovideals} have $\ppk{\J}$ (provably) equal to $\pp$.
\begin{quest}
Is it consistent that $\ppk{\J}>\pp$ for some nicely definable (e.g., Borel, analytic, $\Delta^1_2$, etc.) ideal~$\J$?
\end{quest}
\par
If $\cp{X}$ is an~$\zsubselmgR{\J}$-space, then $X$ is an~$\subselmgR{\J}$-space. The~reversed implication is true for non-\tall\ ideal~$\J$ (being Fr\' echet--Urysohn space and a~$\gamma$-set, respectively, see~\cite{Su20}) as well as in an~$\sonemm$-space (Proposition~\ref{subselplusmm}). 
\begin{quest}\label{q0006}
\begin{enumerate}[(1)]
    \item Is there an~$\subselmgR{\J}$-space~$X$ such that $\cp{X}$ is not an~$\zsubselmgR{\J}$-space?
    \item What if $\J$ is not $\omega$-hitting (and hence $\subselmgR{\J}$-space implies~$\pi$)?
    \item Is it consistent that a~dominating subset of~$\baire\omega$ is a~$\pi$-set?
\end{enumerate}
\end{quest}
\par
The next question is related to the last item of Question~\ref{q0006}.
\begin{quest} \label{q0007}
Are $\pi$-sets of reals preserved by adding a Cohen real?
\end{quest}
Since $\gamma$-sets are preserved by adding Cohen reals by \cite[Theorem~5.1]{MilTsaZdo16}, a  negative answer to Question~\ref{q0007} would 
give a $\pi$-set of reals that is not a $\gamma$-set, with an alternative (forcing-theoretic) proof. On the other hand, an affirmative answer would imply much more than a negative answer to Question~\ref{q0006}(3): After adding uncountably many Cohen  reals each ground model set $X$ of reals is easily seen to become Rothberger, and thus so do  all its finite powers, being also ground model sets of reals. Combining this
with \cite[Theorem~11]{TsZd09} we conclude that all finite powers of 
$X$ are also Hurewicz in the Cohen extension. Now, 
\cite[Theorem~5.2]{MilTsaZdo16} implies that all finite powers of $X$ must have been both Rothberger and Hurewicz already in the~ground model.

\smallskip
Let us recall that a~topological space~$X$ has property $\delta$ if for every  open $\omega$-cover $\V$ of~$X$ there is $\alpha<\omega_1$ such that $X\in\rmL^\alpha(\V)$. On the~other hand, $X$ has property~$\delta'$ if there is $\alpha<\omega_1$ such that for any open $\omega$-cover~$\V$ we have $X\in\rmL^\alpha(\V)$. Thus any $\delta'$-set is a~$\delta$-set as well. However, we do not know whether these two notions are provably equivalent.
\begin{quest}
Is any $\delta$-set also a~$\delta'$-set?
\end{quest}


\vspace{0.5cm}

{\bf Acknowledgement.}
We are deeply committed to the~referee for his proposal for a~revision of the paper, which we believe improved substantially the~presentation. We would like to thank the~participants of Ko\v sice set-theoretical seminar for valuable comments and suggestions.


\end{document}